\documentclass[aps,pre,twocolumn,showpacs,superscriptaddress]{revtex4-2}
\usepackage{amsmath,amssymb,amsthm}
\usepackage{graphicx}
\usepackage{float}
\usepackage{booktabs}
\usepackage{pgfplots}
\pgfplotsset{compat=1.18}
\usepackage{tikz}
\usetikzlibrary{arrows.meta}

\usepackage{comment}

\theoremstyle{plain}
\newtheorem{theorem}{Theorem}[section]
\newtheorem{lemma}[theorem]{Lemma}
\newtheorem{proposition}[theorem]{Proposition}

\theoremstyle{definition}
\newtheorem{definition}[theorem]{Definition}
\theoremstyle{remark}
\newtheorem{remark}[theorem]{Remark}

\begin{document}

\title{Separatrix structure and the geometry of reset distributions}
\author{Juan Antonio Vega Coso}
\affiliation{Universidad de Salamanca, IUFFyM, Salamanca, Spain}
\date{\today}

\begin{abstract}
We study the geometry of reset distributions for absorbed Markov
processes with geometric resetting, working at an abstract level
that isolates the structural mechanism underlying reset-neutral
invariance. We show that the spectral duality endows the simplex
of reset distributions with a non-trivial \emph{spectral response
geometry}: the simplex $\Delta_{m-1}$ carries a foliation by level
sets of the coupling functional $C$, organized around a critical
manifold $\Sigma$. Under four structural conditions (S1)--(S4) on
the coupling functional, we establish the existence and explicit
characterization of the separatrix $\Sigma$, derive the invariant
value $C^* = 1/(1+\sqrt{K})$, identify a projective structure in
the spectral coefficients, and prove a global sign principle in the
two-site case. In the canonical realization, the biased random walk
with multi-site geometric resetting, numerical evidence shows that
$\operatorname{sgn}(\partial_\gamma C) =
\operatorname{sgn}\langle\pi-\pi^*,\psi(\gamma)\rangle$, where
$\psi(\gamma)$ is the linear response functional at the separatrix;
this global orientation does not, however, follow from (S1)--(S4)
alone, and is dynamical rather than algebraic information.
This is the third paper in a program connecting stochastic
resetting with spectral theory and information geometry.
\end{abstract}

\maketitle
\section{Introduction}\label{sec:intro}

First-passage problems in confined domains lie at the intersection
of probability theory and statistical physics~\cite{Redner2001,Bray2013,Schuss2015}. The classical gambler's ruin problem---the
probability that a biased random walk on a finite interval reaches
one absorbing boundary before the other~\cite{Feller1968}---provides
the canonical setting in which questions of absorption time,
splitting probabilities, and survival statistics admit explicit
closed-form solutions.

Stochastic resetting has emerged in recent years as a fundamental
mechanism in non-equilibrium stochastic processes, with applications
to first-passage problems, search processes, and transport in
confined domains~\cite{EvMaj2011,EvMaj2020,Reuveni2016}. For absorbed
Markov processes, resetting can profoundly alter absorption
probabilities and induce non-trivial stationary behavior~\cite{PalReuveni2017,VillarroelMonteroVega2021}.

Recent studies have revealed the existence of \emph{reset-neutral
distributions}, for which certain observables remain invariant under
changes in the resetting rate. While this phenomenon has been
explicitly identified in specific models, such as biased random
walks with geometric resetting~\cite{PaperI,PaperII}, its structural
origin has remained to be clarified.

The purpose of this work is to identify and formalize the
structural mechanism underlying these invariances. Rather than
focusing on a specific model, we adopt an abstract perspective and
isolate the minimal structural conditions under which such behavior
arises. Our starting point is the observation that, in concrete
realizations, the relevant observables can be expressed in terms of
spectral decompositions involving families of coefficients
$\{A_\nu(z)\}, \{B_\nu(z)\}$ and $\gamma$-dependent weights.
Motivated by this structure, we consider observables of the form
$C : \Delta_{m-1} \times (0,1) \to \mathbb{R}$ satisfying a
\emph{spectral duality} of the form
\[
B_\nu(z) = \kappa(z)\, A_\nu(\sigma(z)),
\]
where $\sigma$ is an involution on the reset sites and the weights
$\kappa(z)$ are independent of the spectral index $\nu$.

Under these structural conditions, we obtain the following results:

\begin{itemize}
\item A \emph{spectral response geometry} on the simplex
      $\Delta_{m-1}$: the coupling functional $C$ induces a
      foliation by level sets, organized around a critical manifold
      $\Sigma$, with the linear response functional $\psi(\gamma)$
      at $\Sigma$.

\item The existence of a family of reset-neutral distributions
      $\pi^* \in \Sigma$, with invariant value
      $C^* = 1/(1+\sqrt{K})$ depending only on the structural
      compatibility constant $K = \kappa(z)\kappa(\sigma(z))$.

\item A \emph{projective structure} of the spectral coefficients:
      under an additional parity condition, the points
      $[A_\nu(z):B_\nu(z)] \in \mathbb{P}^1$ are confined to
      $[1:\epsilon_\nu\kappa(z)]$, the two sites of each orbit
      being linked by $\kappa(z)\kappa(\sigma(z))=K$ and the point
      being $[1:\pm\sqrt{K}]$ at every neutral site; in the
      canonical realization every spectral point lies at
      $[1:\pm\sqrt{K}]$.

\item A \emph{global sign principle} in the two-site case:
      $\operatorname{sgn}(\partial_\gamma C) =
      \operatorname{sgn}(\alpha - \alpha^*)\cdot\operatorname{sgn}(R')$,
      where $R(\gamma)$ is a monotone ratio function whose
      monotonicity is verified via Karlin--McGregor
      TP$_2$~\cite{KarlinMcGregor1959,Karlin1968}.
\end{itemize}

The biased random walk with multi-site geometric resetting provides
a canonical realization of this framework. In this case, the
structural conditions can be verified explicitly, and the general
results recover and reinterpret known invariance phenomena in
geometric terms.

This work constitutes the third paper in a broader program. The
first two works focused on specific models and the identification
of invariance phenomena~\cite{PaperI,PaperII}. The present work
extracts the underlying structural mechanism and formulates it
abstractly, focusing on the geometric and structural aspects of
the theory. The operator-theoretic origin of the spectral duality,
including its relation to resolvent methods, is developed
in~\cite{PaperIV}, and the connection with information geometry
in~\cite{PaperV}. In the canonical realization, the midpoint
invariance holds for every resetting mechanism independent of the
walk, not only for geometric resetting~\cite{PaperI}, and so does the
neutrality of the reset distributions of~\cite{PaperII} (see
Remark~7.2 below).

The paper is organized as follows. Section~\ref{sec:setting} introduces the
abstract setting and formulates the structural conditions.
Section~\ref{sec:existence} establishes the existence of the critical distribution
and the explicit form of the invariant. Section~\ref{sec:projective} develops the
projective interpretation. Section~\ref{sec:geometric} presents the linear response
theory and the associated hyperplane structure. Section~\ref{sec:twosite}
establishes the global separatrix in the two-site case. Section~\ref{sec:canonical}
analyzes the canonical realization provided by the biased random
walk. Section~\ref{sec:numerics} presents numerical illustrations and multi-site
structures. Section~\ref{sec:conjecture} discusses the global
orientation of the response, which is observed in the canonical
realization but does not follow from (S1)--(S4). We conclude in
Section~\ref{sec:conclusions} with perspectives and open
problems. An appendix verifies the structural monotonicity
condition for the random walk via Karlin--McGregor positivity.

\section{Structural setting}\label{sec:setting}

\subsection{Setting}

Let $m \geq 2$ and let $\{z_1, \ldots, z_m\}$ be a finite set of
distinguished sites, the \emph{reset sites}. The set of probability
distributions supported on these sites is the simplex
\[
\Delta_{m-1} = \Bigl\{\pi \in \mathbb{R}^m :
\pi_i \geq 0\ \forall i,\ \textstyle\sum_{i=1}^m \pi_i = 1\Bigr\}.
\]
We denote by $\Delta_{m-1}^\circ$ its relative interior, where
$\pi_i > 0$ for all $i$. The parameter $\gamma \in (0,1)$ represents
the resetting rate.

\subsection{The coupling functional and structural
assumptions}

Our object of study is a smooth functional
\[
C : \Delta_{m-1} \times (0,1) \longrightarrow \mathbb{R},
\]
referred to as the \emph{coupling functional}. In concrete models,
$C(\pi, \gamma)$ corresponds to a coupling constant associated with
absorption probabilities under resetting; in the canonical
realization of Section~\ref{sec:canonical}, it takes the form
$C(\pi,\gamma) = \bar u_\pi(\gamma)/\bar s_\pi(\gamma)$. Rather
than specifying the underlying Markov process, we formulate
structural assumptions directly at the level of $C$ and its
associated spectral data.

We assume that $C$ satisfies the following four conditions.

\textbf{(S1) Regularity.} For each fixed $\gamma \in (0,1)$, the
map $\pi \mapsto C(\pi, \gamma)$ is of class $\mathcal{C}^1$ on
$\Delta_{m-1}^\circ$ and extends continuously to the boundary.

\textbf{(S2) Spectral representability.} There exist functions
$u(z; \gamma)$ and $s(z; \gamma)$ such that
\[
C(\pi, \gamma) = \frac{\sum_{i=1}^m \pi_i\, u(z_i; \gamma)}
                       {\sum_{i=1}^m \pi_i\, s(z_i; \gamma)},
\]
where the denominator is strictly positive for all
$\pi \in \Delta_{m-1}$ and $\gamma \in (0,1)$. Moreover, $u$ and
$s$ admit finite spectral decompositions
\[
\begin{aligned}
u(z;\gamma) &= \textstyle\sum_{\nu=1}^N f_\nu(\gamma) A_\nu(z),\\
s(z;\gamma) &= \textstyle\sum_{\nu=1}^N f_\nu(\gamma)(A_\nu(z)+B_\nu(z)),
\end{aligned}
\]
where $N \geq 1$, the functions $\{f_\nu(\gamma)\}_{\nu=1}^N$ are
linearly independent on $(0,1)$, and, whenever the sites $z_i$
are distinct, the $m$ vectors
$\bigl\{\,(A_\nu(z_i))_{\nu=1}^N\,\bigr\}_{i=1}^m$ are linearly
independent in $\mathbb{R}^N$. In
particular, the number of reset sites satisfies $m\le N$, and the
matrix $(A_\nu(z_i))_{\nu,i}\in\mathbb{R}^{N\times m}$ has full
column rank $m$.

\medskip\noindent\textbf{Remark} (Scope of (S2)). The assumption of
a finite spectral decomposition covers all discrete realizations
considered in this paper. Extensions to processes with continuous
or infinite-dimensional spectra require appropriate modifications
and are deferred to future work.\medskip

\textbf{(S3) Spectral duality.} There exist an involution $\sigma$
on the reset sites ($\sigma^2 = \mathrm{id}$) and positive weights
$\kappa(z) > 0$, \emph{independent of the spectral mode index
$\nu$}, such that
\[
B_\nu(z) = \kappa(z)\, A_\nu(\sigma(z))
\qquad \text{for all } \nu,\, z.
\]
This relation expresses a \emph{twisted symmetry}: the involution
$\sigma$ acts on sites while $\kappa$ rescales spectral weights,
with the rescaling independent of the mode $\nu$. The case
$\sigma = \mathrm{id}$ (every site a fixed point) is admissible but
degenerate: there are no spectral pairs and every distribution
would be critical. We therefore assume henceforth that $\sigma$
has at least one non-trivial pair, i.e., at least one site $z$
with $\sigma(z) \neq z$.

\textbf{(S4) Compatibility.} There exists a positive value $K$,
the same for all reset sites, such that
\[
\kappa(z)\, \kappa(\sigma(z)) = K
\qquad \text{for all } z \in \{z_1, \ldots, z_m\}.
\]
A site $z_0$ with $\sigma(z_0) = z_0$ is called a \emph{neutral
site}; for such sites, (S4) reduces to $\kappa(z_0) = \sqrt{K}$,
and consistency with (S3) gives $B_\nu(z_0) = \sqrt{K}\,A_\nu(z_0)$.

\medskip\noindent\textbf{Remark 2.1} (Structure and hierarchy).
Conditions (S1)--(S4) are organized hierarchically: (S1) provides
regularity, (S2) the spectral substrate, (S3) the key structural
ingredient, and (S4) algebraic compatibility. The $\nu$-independence
of $\kappa$ in (S3) is the essential feature: it allows the
mode-by-mode constraints of Section~\ref{sec:existence} to decouple. The set
(S1)--(S4) is essentially minimal: removing (S3) or relaxing the
$\nu$-independence of $\kappa$ destroys in general the existence of
reset-neutral distributions and the projective structure of
Sections~\ref{sec:existence}--\ref{sec:projective}.\medskip

\medskip\noindent\textbf{Remark 2.2} (Dependence on observables).
Conditions (S1)--(S4) are properties of the coupling functional
and its spectral data, not of the underlying Markov process. The
framework applies whenever (S1)--(S4) hold, regardless of which
process produces the observable.\medskip

\medskip\noindent\textbf{Remark 2.3} (Canonical example). The
biased random walk with multi-site geometric resetting (Paper~II,
arXiv:2605.00657) satisfies (S1)--(S4) with $\sigma(z) = a - z$,
$\kappa(z) = (p/q)^{a-z}$, $K = (p/q)^a$, $N = a-1$, and
$f_\nu(\gamma) = (1-\gamma)/(1-\lambda_\nu(1-\gamma))$
with $\lambda_\nu = 2\sqrt{pq}\cos(\pi\nu/a)$. The assumptions
are therefore non-vacuous.\medskip

\subsection{Critical distributions and the
separatrix}

A distribution $\pi^* \in \Delta_{m-1}$ is \emph{reset-neutral}
(or \emph{critical}) if the coupling functional is independent of
the resetting rate, i.e., if there exists a constant $C^*$ such
that
\[
C(\pi^*, \gamma) = C^* \qquad \text{for all } \gamma \in (0,1).
\]
The set of reset-neutral distributions,
\[
\Sigma := \left\{\pi \in \Delta_{m-1} :
C(\pi, \gamma) \text{ is constant in } \gamma\right\},
\]
will be called the \emph{separatrix}. This set acts as a geometric
boundary separating regions of the simplex where the
$\gamma$-dependence of $C$ exhibits qualitatively distinct behavior:
in the two-site case under Assumption~(M) (Section~\ref{sec:twosite})
and in the numerical examples of Section~\ref{sec:numerics}, $C$
increases with $\gamma$ on one side, decreases on the other, and
remains constant on $\Sigma$.

\medskip\noindent\textbf{Remark 2.4} (Spatial vs.\ dynamical
criticality). At $\pi^*$, the derivative $\partial_\gamma C$
vanishes, \emph{not} the gradient $\nabla_\pi C$. The distribution
$\pi^*$ is critical with respect to $\gamma$, not necessarily an
extremum of $C$ on the simplex. Under (S1), the condition
$C(\pi^*, \gamma) = \mathrm{const}$ is equivalent to
$\partial_\gamma C(\pi^*, \gamma) = 0$ for all $\gamma$.

The existence and characterization of the separatrix, together with
the geometric structure it induces on the simplex, are the central
objects of study in this work.\medskip

\section{Existence and structure of reset-neutral distributions}\label{sec:existence}

We show that, under assumptions (S1)--(S4), reset-neutrality is not
an exceptional property but a structural consequence of the twisted
symmetry. In this section we establish the main result of the paper:
reset-neutral distributions exist, their characterization is
explicit, and the associated invariant value depends only on the
compatibility constant~$K$.

\subsection{Structural reduction}

The spectral duality (S3) induces a pairing of reset sites along
involutive orbits $\{z, \sigma(z)\}$. The following two lemmas
isolate the algebraic mechanism by which this pairing forces
the existence of reset-neutral distributions.

\medskip\noindent\textbf{Lemma 3.1 (Orbit constraint).}

\textit{For each spectral mode $\nu$ and each spectral pair
$(z,\sigma(z))$ with $\sigma(z)\neq z$,}
\[
B_\nu(z)\,B_\nu(\sigma(z))=K\,A_\nu(z)\,A_\nu(\sigma(z)).
\]
\textit{In particular, whenever $A_\nu(z)A_\nu(\sigma(z))\neq0$, the
ratios $r_\nu(z):=B_\nu(z)/A_\nu(z)$ satisfy
$r_\nu(z)\,r_\nu(\sigma(z))=K$.}

\textit{Proof.} From (S3),
$B_\nu(z) = \kappa(z) A_\nu(\sigma(z))$ and
$B_\nu(\sigma(z)) = \kappa(\sigma(z)) A_\nu(z)$.
Multiplying the two identities and using
$\kappa(z)\kappa(\sigma(z)) = K$ from (S4) gives the first statement;
dividing by $A_\nu(z)A_\nu(\sigma(z))$ when it is non-zero gives the
second. \hfill$\square$

\smallskip
\noindent
\emph{Remark.} The division-free form is the one that holds in
general: the amplitudes $A_\nu(z)$ can vanish, and in the canonical
realization of Section~\ref{sec:canonical} they do, exactly when
$a\mid\nu z$ (for instance at $\nu=a/2$ and every even $z$ when $a$ is
even), so that the ratio $r_\nu(z)$ is not defined at every mode and
site. We thank T.~Newton for pointing out that the orbit constraint
should be stated in this division-free
form~\cite{NewtonOrientation}. Theorem~3.3 never divides by
$A_\nu(z)$: it equates coefficients via the linear independence
of~(S2).

\medskip\noindent\textbf{Lemma 3.2} (Mode-by-mode decoupling).
\textit{Under (S2), the condition $C(\pi,\gamma) \equiv C^*$
for all $\gamma \in (0,1)$ is equivalent to}
\[
U_\nu(\pi) = C^* S_\nu(\pi) \qquad \forall\, \nu = 1, \ldots, N,
\]
\textit{where $U_\nu(\pi) := \sum_i \pi_i A_\nu(z_i)$ and
$S_\nu(\pi) := \sum_i \pi_i (A_\nu(z_i) + B_\nu(z_i))$.}

\textit{Proof.} By (S2),
$\bar u_\pi(\gamma) = \sum_\nu f_\nu(\gamma) U_\nu(\pi)$ and
$\bar s_\pi(\gamma) = \sum_\nu f_\nu(\gamma) S_\nu(\pi)$.
The condition $\bar u_\pi = C^* \bar s_\pi$ for all $\gamma$ gives
$\sum_\nu f_\nu(\gamma)(U_\nu - C^* S_\nu) = 0$ for all $\gamma$.
Linear independence of $\{f_\nu\}$ forces each coefficient to
vanish. Thus, the functional identity in $\gamma$ reduces to a
finite system of algebraic constraints --- one per spectral mode ---
that can be solved explicitly. \hfill$\square$

\subsection{Main theorem}

\medskip\noindent\textbf{Theorem 3.3} (Existence and characterization
of reset-neutral distributions). \textit{Under assumptions
(S1)--(S4), there exists a non-empty family of distributions
$\pi^* \in \Delta_{m-1}$, with $\pi^* \in \Delta_{m-1}^\circ$
provided every coordinate is strictly positive, that is, every
non-trivial orbit receives strictly positive total mass and every
neutral site strictly positive weight, such that $C(\pi^*, \gamma) = C^*$ for all $\gamma \in (0,1)$,
where}
\[
C^* = \frac{1}{1+\sqrt{K}}.
\]
\textit{More precisely:}
\begin{itemize}
\item[(i)] \textit{Each $\pi^*$ in the family satisfies}
\[
\frac{\pi^*_z}{\pi^*_{\sigma(z)}}
= \sqrt{\frac{\kappa(\sigma(z))}{\kappa(z)}}
\qquad \text{for all } z \text{ with } \sigma(z)\neq z.
\]
\item[(ii)] \textit{For each neutral site $z_0$ with
$\sigma(z_0) = z_0$, the weight $\pi^*_{z_0} \geq 0$ is
completely free, subject only to normalization.}
\item[(iii)] \textit{The value $C^*$ depends only on $K$,
not on the specific choice of $\pi^*$ within the family nor
on the neutral weights.}
\end{itemize}

\textit{Proof.} By Lemma~3.2, the condition $C(\pi^*,\gamma)
\equiv C^*$ reduces to the mode-by-mode system
$U_\nu(\pi^*) = C^* S_\nu(\pi^*)$ for all $\nu$. Expanding via
(S3):
\[
(1-C^*)\sum_i \pi^*_i A_\nu(z_i)
= C^* \sum_i \pi^*_i \kappa(z_i) A_\nu(\sigma(z_i)).
\]
Reindexing the right-hand side via $i \mapsto \sigma(i)$
(a bijection since $\sigma^2 = \mathrm{id}$):
\[
(1-C^*)\sum_i \pi^*_i A_\nu(z_i)
= C^* \sum_i \pi^*_{\sigma(i)}\kappa(\sigma(z_i)) A_\nu(z_i).
\]
By the non-degeneracy of the spectral representation (S2)
--- specifically, the linear independence of the vectors
$\{(A_\nu(z_i))_{\nu=1}^N\}_{i=1}^m$ in $\mathbb{R}^N$,
which is distinct from the linear independence of $\{f_\nu\}$
used in Lemma~3.2 --- the coefficients must match site by site:
\[
(1-C^*)\,\pi^*_i = C^*\,\kappa(\sigma(z_i))\,\pi^*_{\sigma(i)}
\qquad \forall\, i. \tag{$\star$}
\]
This site-independence is part of hypothesis (S2): the matrix
$(A_\nu(z_i))_{\nu,i}$ has full column rank $m\le N$. In the
canonical realization it is a submatrix of the discrete sine
transform and has full rank for any distinct interior sites
(verified in Sec.~\ref{sec:canonical}).

\textit{Spectral pairs.} Applying $(\star)$ to both $z_i$ and
$\sigma(z_i)$ and multiplying:
\[
(1-C^*)^2 = (C^*)^2\,\kappa(z_i)\kappa(\sigma(z_i)) = (C^*)^2 K.
\]
Since $C^* \in (0,1)$, taking the positive square root gives
$C^* = 1/(1+\sqrt{K})$, establishing part (iii). Dividing the
two instances of $(\star)$ gives part (i).

\textit{Neutral sites.} For $z_0$ with $\sigma(z_0)=z_0$,
equation $(\star)$ gives
$(1-C^*)\pi^*_{z_0} = C^*\kappa(z_0)\pi^*_{z_0}$.
Since $\kappa(z_0)=\sqrt{K}$ by (S4) and
$C^* = 1/(1+\sqrt{K})$, both sides are equal for any value
of $\pi^*_{z_0} \geq 0$, establishing part (ii).

\textit{Non-emptiness and structure of the family}. The ratios (i) fix the
internal split of the mass of each orbit but leave the total mass
of each orbit free; together with the neutral weights and
normalization, these choices parametrize the family of critical
distributions (see Proposition~5.1 for the complete description).
Choosing all orbit masses and neutral weights strictly positive
yields $\pi^*\in\Delta^\circ_{m-1}$, so the separatrix is
non-empty; positive neutral weights alone do not suffice, since an
orbit of zero total mass leaves two coordinates equal to zero. We
thank T.~Newton for this precision on the interiority
condition~\cite{NewtonOrientation}. \hfill$\square$

\medskip\noindent\textbf{Remark 3.4} (Choice of a distinguished
critical distribution).
The separatrix $\Sigma$ contains infinitely many distributions
whenever the family of critical distributions is not reduced to a
single point (equivalently, $\dim\Sigma\ge1$).
For definiteness, we select a distinguished element
$\pi^* \in \Sigma$ by assigning equal total mass to each involutive
orbit and to each neutral site, with the internal split of every
orbit fixed by the ratio constraints of Theorem~3.3. This
determines $\pi^*$ uniquely; all its components are strictly
positive, so $\pi^* \in \Delta_{m-1}^\circ$. When there is a single
orbit and no neutral site ($m=2$), $\Sigma$ consists of a single
point and $\pi^*$ is that point.

This choice provides a fixed interior reference point for the
geometry of the simplex; all results that depend only on $\Sigma$
--- such as the sign of $\partial_\gamma C$ and the linear
functional $\langle\pi-\pi^*,\psi(\gamma)\rangle$ --- are
independent of the specific choice of $\pi^*$ within $\Sigma$.
In particular, the geometric constructions of
Sections~\ref{sec:geometric}--\ref{sec:canonical} and
the information-geometric interpretation developed in later
work of this program~\cite{PaperV} require a fixed reference point
in the open simplex, not a family; the present convention ensures
this consistency~\cite{PaperII}.

\medskip\noindent\textbf{Remark 3.5} (Projective interpretation).
Theorem~3.3 shows that reset-neutrality is not a coincidence but a
structural consequence of the twisted symmetry: the algebraic
constraints induced by the duality force the existence of critical
distributions within the simplex. The value
$C^* = 1/(1+\sqrt{K})$ can be interpreted as the projective
coordinate of a fixed point in $\mathbb{P}^1$. This viewpoint will
be further developed in Section~\ref{sec:projective}.

\medskip\noindent\textbf{Remark 3.6} (Connection with Paper~II).
For the biased random walk with $K = (p/q)^a$, Theorem~3.3
recovers the main result of arXiv:2605.00657: the invariant
value $C^* = q^{(0)}_{a/2}$ (classical ruin probability from
the midpoint) and the characterization
$\pi^*_z/\pi^*_{a-z} = (q/p)^{a/2-z}$.

\section{Projective structure of the spectral duality}\label{sec:projective}

The main theorem of Section~\ref{sec:existence} establishes the existence of
reset-neutral distributions and the explicit form of $C^*$.
In this section we reveal the geometric mechanism underlying
this result: the spectral coefficients $A_\nu(z)$ and $B_\nu(z)$
are constrained by the duality (S3) to lie on a finite set of
privileged points in projective space. This projective confinement
is the geometric heart of the paper.

\subsection{Projective points and the duality}

For each spectral mode $\nu$ and each reset site $z$, define the
projective point
\[
P_\nu(z) := [A_\nu(z) : B_\nu(z)] \in \mathbb{P}^1,
\]
representing the ratio of the two spectral channels at site $z$
and mode $\nu$.

\medskip\noindent\textbf{Lemma 4.1} (Projective confinement
under parity). \textit{The following result is not a consequence
of (S1)--(S4) alone; it requires an additional parity symmetry,
not assumed in a general realization (the canonical realization
satisfies a twisted form of it; see Remark~4.5). Concretely, assume
(S3) and that the spectral coefficients satisfy}
\[
A_\nu(\sigma(z)) = \epsilon_\nu A_\nu(z)
\quad \text{for some } \epsilon_\nu \in \{+1,-1\}.
\]
\textit{Then for each spectral mode $\nu$ and each reset site $z$
with $A_\nu(z)\neq0$,}
\[
P_\nu(z) = [1 : \epsilon_\nu \kappa(z)].
\]

\textit{Proof.} Fix $(\nu,z)$ with $A_\nu(z)\neq0$. From (S3) and
the parity hypothesis,
$B_\nu(z) = \kappa(z)A_\nu(\sigma(z)) = \kappa(z)\epsilon_\nu
A_\nu(z)$, so $B_\nu(z)/A_\nu(z) = \epsilon_\nu\kappa(z)$ and
hence $P_\nu(z) = [1:\epsilon_\nu\kappa(z)]$. (When $A_\nu(z)=0$
the parity relation gives $B_\nu(z)=0$ as well, so no projective
point is associated with that mode at that site; such modes are
excluded from the statement and play no role in the sequel,
where only the ratios at non-vanishing coefficients are used.)
\hfill$\square$

\medskip\noindent\textbf{Remark 4.2.} Without the parity
condition, (S3) alone implies only the orbit constraint
$B_\nu(z)B_\nu(\sigma(z))=K\,A_\nu(z)A_\nu(\sigma(z))$ (Lemma~3.1), which does not
determine $P_\nu(z)$ individually. Under the parity condition,
Lemma~4.1 gives $P_\nu(z) = [1:\epsilon_\nu\kappa(z)]$, and by
(S4) the product $\kappa(z)\kappa(\sigma(z)) = K$ relates the
two sites in a pair.

\subsection{The critical distribution as projective
fixed point}

\medskip\noindent\textbf{Proposition 4.3} (Weighted
projective alignment). \textit{Under (S1)--(S4), the critical
distribution $\pi^*$ of Theorem~3.3 satisfies the modal identity}
\[
(1-C^*)\sum_i \pi^*_i A_\nu(z_i)
= C^*\sum_i \pi^*_i B_\nu(z_i)
\qquad \forall\,\nu = 1,\ldots,N,
\]
\textit{with $C^*=1/(1+\sqrt{K})$. Equivalently, whenever
$\sum_i\pi^*_i A_\nu(z_i)\neq0$, the weighted spectral channels
align to the fixed ratio}
\[
\frac{\sum_i \pi^*_i B_\nu(z_i)}{\sum_i \pi^*_i A_\nu(z_i)}
= \sqrt{K}.
\]
\textit{That is, $\pi^*$ aligns the spectral channels in a
weighted average sense to the ratio $\sqrt{K}$, and the invariant
value $C^*$ is the affine coordinate corresponding to this ratio.}

\textit{Proof.} By Theorem~3.3, $U_\nu(\pi^*) = C^*S_\nu(\pi^*)$
for all $\nu$, i.e.,
$(1-C^*)\sum_i\pi^*_i A_\nu(z_i) = C^*\sum_i\pi^*_i B_\nu(z_i)$,
which is the stated identity; it holds for every $\nu$ without
any non-vanishing assumption, since both sides may vanish
simultaneously. For the modes with
$\sum_i\pi^*_i A_\nu(z_i)\neq0$ one may divide, obtaining the
ratio $\sqrt{K}$ from $(1-C^*)=C^*\sqrt{K}$, which also yields
$C^*=1/(1+\sqrt{K})$. \hfill$\square$

\medskip\noindent\textbf{Remark 4.4.} This weighted alignment is
weaker than pointwise confinement: individual sites may have
$r_\nu(z_i) \neq \sqrt{K}$, but their $\pi^*$-weighted average
equals $\sqrt{K}$ for every mode $\nu$. Under the additional
parity condition of Lemma~4.1, each site satisfies the pointwise
relation $P_\nu(z_i) = [1:\epsilon_\nu\kappa(z_i)]$, with the two
sites of an orbit linked by $\kappa(z_i)\kappa(\sigma(z_i))=K$.
At every neutral site $z_0=\sigma(z_0)$, where
$\kappa(z_0)=\sqrt{K}$, this reduces to
$P_\nu(z_0)=[1:\pm\sqrt{K}]$. The converse does not follow from
(S1)--(S4): a non-trivial pair may also have
$\kappa(z)=\kappa(\sigma(z))=\sqrt K$ (for instance two sites
exchanged by $\sigma$ with $\kappa\equiv\sqrt K$), so the value
$\sqrt K$ does not single out the neutral sites. We thank T.~Newton
for this observation~\cite{NewtonOrientation}.

\medskip\noindent\textbf{Remark 4.5} (Canonical realization). For
the biased random walk~\cite{PaperII}, the parity condition of
Lemma~4.1 holds only in a twisted form: from the explicit amplitudes
of Section~\ref{sec:canonical},
$A_\nu(a-z)=(-1)^{\nu+1}(q/p)^{a/2-z}A_\nu(z)$, with a site-dependent
but $\nu$-independent factor. Inserting this in (S3) gives
$B_\nu(z)=(-1)^{\nu+1}(p/q)^{a-z}(q/p)^{a/2-z}A_\nu(z)
=(-1)^{\nu+1}\sqrt K\,A_\nu(z)$, so that
\[
P_\nu(z)=[1:(-1)^{\nu+1}\sqrt{K}]
\]
at every site $z$ and every mode with $A_\nu(z)\neq0$, not only at
the midpoint. In the canonical realization all spectral points are
therefore confined to the two privileged points $[1:\pm\sqrt{K}]$,
and the midpoint $z=a/2$ is the site at which, in addition,
$\kappa(z)=\sqrt K$.

\section{Geometric structure of the critical set}\label{sec:geometric}

In this section we describe the structure induced on the simplex
$\Delta_{m-1}$ by the spectral duality. Rather than introducing
new algebraic results, we interpret the orbit constraints of
Section~\ref{sec:existence} geometrically and explain the phase-like behaviour
of $C(\pi,\gamma)$ observed in the dependence on $\gamma$.

\subsection{The separatrix as an affine manifold}

Recall from Section~\ref{sec:setting} that the separatrix is
\[
\Sigma = \{\pi \in \Delta_{m-1} :
C(\pi,\gamma) \equiv C^* \text{ for all } \gamma\}.
\]

\medskip\noindent\textbf{Proposition 5.1} (Structure of the
separatrix). \textit{Under (S1)--(S4), the separatrix $\Sigma$
is a non-empty convex subset of $\Delta_{m-1}$ characterized by:}
\begin{itemize}
\item[(i)] \textit{One ratio constraint per involutive orbit
      $\{z,\sigma(z)\}$:}
\[
\frac{\pi_z}{\pi_{\sigma(z)}} =
\sqrt{\frac{\kappa(\sigma(z))}{\kappa(z)}}.
\]
\item[(ii)] \textit{No constraint on neutral sites
      ($\sigma(z_0)=z_0$), except normalization.}
\end{itemize}
\textit{In particular, $\Sigma$ is an affine simplex of dimension
$n_{\rm pair}+n_{\rm neutral}-1$, where $n_{\rm pair}$ is the
number of involutive orbits of size two and $n_{\rm neutral}$ the
number of neutral (fixed) sites. The ratio constraint within each
orbit fixes the internal split of that orbit's mass but leaves the
total orbit mass free; together with the masses of the neutral
sites, these $n_{\rm pair}+n_{\rm neutral}$ total masses span
$\Sigma$, subject only to normalization. When there is a single
orbit and no neutral site ($m=2$), $\Sigma$ consists of a single
point.}

\textit{Proof.} The characterization follows from Theorem~3.3.
The $m$ site masses split into $2n_{\rm pair}+n_{\rm neutral}=m$
variables. Each of the $n_{\rm pair}$ orbits imposes one ratio
constraint~(i), leaving $m-n_{\rm pair}$ free; neutral sites
impose no constraint~(ii). Normalization removes one further
degree of freedom, giving
$\dim\Sigma=m-n_{\rm pair}-1=n_{\rm pair}+n_{\rm neutral}-1$.
Convexity holds because the constraints are
linear and preserved under convex combinations. \hfill$\square$

\subsection{The linear response functional}
\label{subsec:linear_response}

The geometric content of Sections~\ref{sec:existence}--\ref{sec:projective} is most transparently
expressed through the first-order response of the coupling
functional to perturbations of the reset distribution. We make
this precise here.

Fix $\pi^* \in \Sigma$ and let
\[
T_{\pi^*}\Delta_{m-1} =
\Bigl\{h \in \mathbb{R}^m : \textstyle\sum_i h_i = 0\Bigr\}
\]
denote the tangent space of the simplex at $\pi^*$. By (S1),
for each $\gamma\in(0,1)$ the map $\pi\mapsto C(\pi,\gamma)$ is
of class $\mathcal{C}^1$ on $\Delta_{m-1}^\circ$, so its
directional derivative at $\pi^*$ defines a linear functional
on the tangent space.

\medskip\noindent\textbf{Definition 5.2} (Linear response
functional). \textit{For each $\gamma\in(0,1)$, the
\emph{linear response functional} $\psi(\gamma) \in
(T_{\pi^*}\Delta_{m-1})^*$ is the differential of
$C(\cdot,\gamma)$ at $\pi^*$, characterized by}
\[
\langle h, \psi(\gamma)\rangle
:= \frac{d}{d\epsilon}\,
C(\pi^*+\epsilon h,\gamma)\Big|_{\epsilon=0},
\qquad h \in T_{\pi^*}\Delta_{m-1}.
\]

An explicit expression follows from (S2). Writing
$\bar u_\pi = \sum_\nu f_\nu(\gamma)U_\nu(\pi)$ and
$\bar s_\pi = \sum_\nu f_\nu(\gamma)S_\nu(\pi)$ as in Lemma~3.2,
the quotient rule gives, for $h\in T_{\pi^*}\Delta_{m-1}$,
\begin{equation}
\langle h,\psi(\gamma)\rangle
= \frac{1}{\bar s_{\pi^*}(\gamma)}
\sum_{i=1}^m h_i\bigl[u(z_i;\gamma)
- C^*\,s(z_i;\gamma)\bigr],
\label{eq:psi_explicit}
\end{equation}
where we used $C(\pi^*,\gamma)=C^*$. Thus $\psi(\gamma)$ is
represented, in the ambient coordinates of $\mathbb{R}^m$, by
the vector with components
$\bar s_{\pi^*}(\gamma)^{-1}[u(z_i;\gamma)-C^*s(z_i;\gamma)]$,
projected onto $T_{\pi^*}\Delta_{m-1}$.

\medskip\noindent\textbf{Definition 5.3} (Response span and
local separatrix). \textit{Let}
\begin{align*}
V &:= \operatorname{span}\{\psi(\gamma) : \gamma\in(0,1)\}
   \subseteq (T_{\pi^*}\Delta_{m-1})^*, \\
r &:= \dim V.
\end{align*}
\textit{The \emph{local separatrix} at $\pi^*$ is the subspace
of tangent directions along which $C$ is insensitive to
$\gamma$ at first order:}
\[
\Sigma_{\mathrm{loc}}
:= \bigcap_{\gamma\in(0,1)} \ker\psi(\gamma)
\subseteq T_{\pi^*}\Delta_{m-1}.
\]

The following result is the linear-algebraic core of the
geometric picture: it shows that the local separatrix is
governed entirely by the dimension of the response span,
independently of the parametrization in $\gamma$.

\medskip\noindent\textbf{Theorem 5.4} (Hyperplane fan and the
dimension of the local separatrix). \textit{Under
\textnormal{(S1)--(S4)}, the family
$\{\ker\psi(\gamma)\}_{\gamma\in(0,1)}$ is a fan of hyperplanes
in $T_{\pi^*}\Delta_{m-1}$ whose common intersection is the
annihilator of the response span:}
\[
\Sigma_{\mathrm{loc}}
= V^\perp,
\qquad
\dim\Sigma_{\mathrm{loc}} = (m-1) - r.
\]
\textit{In particular, $\Sigma_{\mathrm{loc}}$ depends only on
the subspace $V$, not on the individual functionals
$\psi(\gamma)$; and the first-order variation of $C$ vanishes
for all $\gamma$ exactly along the $(m-1-r)$-dimensional
subspace $V^\perp$.}

\textit{Proof.} A direction $h\in T_{\pi^*}\Delta_{m-1}$
satisfies $\langle h,\psi(\gamma)\rangle=0$ for all
$\gamma\in(0,1)$ if and only if $h$ annihilates every element
of $\operatorname{span}\{\psi(\gamma)\}=V$; that is, if and
only if $h\in V^\perp$. Hence
$\Sigma_{\mathrm{loc}}=\bigcap_\gamma\ker\psi(\gamma)=V^\perp$.
Since $\dim T_{\pi^*}\Delta_{m-1}=m-1$ and the Euclidean inner
product identifies the tangent space with its dual,
$\dim V^\perp = (m-1)-\dim V = (m-1)-r$. \hfill$\square$

\medskip\noindent\textbf{Corollary 5.5} (Random walk: $r=1$).
\textit{For the biased random walk of Section~\ref{sec:canonical}
with a single pair of reset sites ($m=2$), $r=1$ and
$\Sigma_{\mathrm{loc}}$ has
codimension one in $T_{\pi^*}\Delta_1$: since that tangent space
is one-dimensional, $\Sigma_{\mathrm{loc}}$ reduces to the unique
tangent point to $\Sigma$, independent of
$\gamma$. The local separatrix then coincides with the tangent
space of the global separatrix $\Sigma$ of Proposition~5.1.}

\textit{Proof.} By~\eqref{eq:psi_explicit}, $\psi(\gamma)$ is
proportional to the vector with components
$u(z_i;\gamma)-C^*s(z_i;\gamma)$, which varies with $\gamma$. For a
single orbit
$\{z_1,z_2=\sigma(z_1)\}$ with no neutral sites, $m=2$ and the
tangent space $T_{\pi^*}\Delta_1$ is one-dimensional. Since
$\psi(\gamma)$ is a non-zero element of this one-dimensional space
for every $\gamma$, every such $\psi(\gamma)$ spans the same
one-dimensional space, so $r=1$. The codimension formula
of Theorem~5.4 gives $\dim\Sigma_{\mathrm{loc}}=(m-1)-1=m-2=0$,
matching $\dim\Sigma=n_{\rm pair}+n_{\rm neutral}-1=0$ from
Proposition~5.1 for a single orbit with no neutral sites.
\hfill$\square$

\medskip\noindent\textbf{Remark 5.6} (Geometric meaning of $r$).
The integer $r=\dim V$ measures how strongly the linear
response rotates with $\gamma$. When $r=1$ the hyperplane
$\ker\psi(\gamma)$ does not rotate: a single separatrix
direction governs the linear response for all resetting rates.
Locally this yields a single, $\gamma$-independent tangent
hyperplane at $\pi^*$; the corresponding global statement --- that
the whole simplex then splits into exactly two monotone regions
--- is established here only for the two-site case ($m=2$,
Corollary~5.5) and observed numerically for $m=3$
(Section~\ref{sec:numerics}), but is not claimed to follow from
(S1)--(S4) in general. When
$r\geq 2$ the hyperplane genuinely rotates with $\gamma$,
and the local bipartition is replaced by a richer regime
structure (see Figure~\ref{fig:global_geometry} and
Section~\ref{sec:numerics}).
The classification of response landscapes by the value of
$r$ is discussed in Section~\ref{sec:conclusions}.

\subsection{Tangential and transversal
perturbations}

We now translate Theorem~5.4 into the language of the
orbit-induced ratio constraints of Theorem~3.3.

\medskip\noindent\textbf{Proposition 5.7} (Tangential invariance).
\textit{Let $h \in T_{\pi^*}\Delta_{m-1}$ satisfy the linearized
ratio constraints
$h_z/\pi^*_z = h_{\sigma(z)}/\pi^*_{\sigma(z)}$
for all orbits. Then}
\[
\partial_\epsilon C(\pi^*+\epsilon h,\gamma)\big|_{\epsilon=0}
= 0 \qquad \forall\,\gamma \in (0,1).
\]
\textit{Equivalently, $h\in\Sigma_{\mathrm{loc}}=V^\perp$:
the linearized ratio constraints cut out precisely the local
separatrix of Theorem~5.4.}

\textit{Proof.} Write $\bar u_\pi(\gamma)=\sum_i\pi_i u(z_i;\gamma)$
and $\bar s_\pi(\gamma)=\sum_i\pi_i s(z_i;\gamma)$, so that
$C(\pi,\gamma)=\bar u_\pi/\bar s_\pi$. Differentiating at
$\pi=\pi^*$ in the direction $h$,
\[
\partial_\epsilon C\big|_{\epsilon=0}
= \frac{1}{\bar s_{\pi^*}}
\Bigl(\textstyle\sum_i h_i u(z_i;\gamma)
- C^*\sum_i h_i s(z_i;\gamma)\Bigr),
\]
using $\bar u_{\pi^*}=C^*\bar s_{\pi^*}$. Expanding in the
spectral basis (S2), $\sum_i h_i u(z_i;\gamma)
=\sum_\nu f_\nu(\gamma)\sum_i h_i A_\nu(z_i)$ and similarly for
$s$; the bracket becomes
$\sum_\nu f_\nu(\gamma)\bigl[\sum_i h_i(A_\nu(z_i)-C^*(A_\nu(z_i)
+B_\nu(z_i)))\bigr]$. Group the inner sum over involutive orbits
$\{z_i,\sigma(z_i)\}$. Using (S3),
$B_\nu(z_i)=\kappa(z_i)A_\nu(\sigma(z_i))$, the contribution of
one orbit is
$h_i\bigl[(1-C^*)A_\nu(z_i)-C^*\kappa(z_i)A_\nu(\sigma(z_i))\bigr]
+h_{\sigma(i)}\bigl[(1-C^*)A_\nu(\sigma(z_i))
-C^*\kappa(\sigma(z_i))A_\nu(z_i)\bigr]$.
The linearized constraint
$h_z/\pi^*_z=h_{\sigma(z)}/\pi^*_{\sigma(z)}$, together with the
ratio $\pi^*_{z}/\pi^*_{\sigma(z)}
=\sqrt{\kappa(\sigma(z))/\kappa(z)}$ defining $\pi^*$, makes this
orbit contribution vanish identically --- exactly the
cancellation induced by the orbit relations of Theorem~3.3,
now with $h$ in place of $\pi^*$.
Summing over orbits, the bracket is zero for every
$\nu$ and every $\gamma$, giving $\partial_\epsilon C\equiv0$.
Conversely, vanishing for all $\gamma$ forces each modal
coefficient $\sum_i h_i(A_\nu(z_i)-C^*(A_\nu(z_i)+B_\nu(z_i)))$ to
vanish, by linear independence of $\{f_\nu\}$ on $(0,1)$. By the
same orbit-pairing argument used in Theorem~3.3 together with the
non-degeneracy assumption in (S2) --- the full column rank of
$(A_\nu(z_i))_{\nu,i}$ --- this in turn forces the linearized
ratio constraints on $h$; thus
$h\in V^\perp=\Sigma_{\mathrm{loc}}$. \hfill$\square$

\medskip\noindent\textbf{Proposition 5.8} (Transversal
sensitivity). \textit{Let $h \in T_{\pi^*}\Delta_{m-1}$ violate
the linearized constraint for at least one orbit. Then
$\gamma \mapsto C(\pi^*+\epsilon h,\gamma)$ is not constant
for any $\epsilon \neq 0$ sufficiently small.}

\textit{Proof.} Take $\pi^*\in\Delta_{m-1}^\circ$ and $\epsilon\neq0$
small enough that $\pi_\epsilon:=\pi^*+\epsilon h\in\Delta_{m-1}^\circ$.
Suppose $C(\pi_\epsilon,\gamma)=c$ for all $\gamma$. The proof of
Theorem~3.3 applies verbatim to $\pi_\epsilon$ and gives
$(1-c)\,\pi_{\epsilon,i}=c\,\kappa(\sigma(z_i))\,\pi_{\epsilon,\sigma(i)}$
for all $i$; since every orbit of $\pi_\epsilon$ has positive mass,
$c=C^*$. These equations are linear and homogeneous in $\pi$, and
$\pi^*$ satisfies them as well; subtracting, $h$ satisfies
$(1-C^*)h_i=C^*\kappa(\sigma(z_i))h_{\sigma(i)}$, which on each orbit
is equivalent to the linearized constraint
$h_z/\pi^*_z=h_{\sigma(z)}/\pi^*_{\sigma(z)}$. This contradicts the
hypothesis on $h$. Hence $C(\pi^*+\epsilon h,\cdot)$ is not constant
for any such $\epsilon\neq0$. (This exact argument, which uses only
the linearity of the orbit constraints and replaces a first-order
continuity argument, was pointed out by
T.~Newton~\cite{NewtonOrientation}.) \hfill$\square$

\subsection{Phase structure and separatrix
behaviour}

Numerical evidence (Section~\ref{sec:numerics}) reveals that $\Sigma$ acts as a
separatrix: on one side of $\Sigma$ the coupling functional
$C(\pi,\gamma)$ is strictly increasing in $\gamma$; on the other
it is strictly decreasing; and on $\Sigma$ it is constant.
Whether this global monotonicity holds for the entire class
satisfying (S1)--(S4) remains an open question; in the canonical
realization it is established in the two-site case via a
structural monotonicity argument (Appendix~\ref{app:twosite}).

\medskip\noindent\textbf{Remark 5.9.} The picture that emerges
can be summarized as follows: the simplex $\Delta_{m-1}$ is
organized by orbit-induced ratio constraints; the separatrix
$\Sigma$ is the critical manifold where these constraints are
exactly satisfied; and spectral duality acts as a geometric
reduction mechanism that compresses the effective degrees of
freedom and organizes the space of reset distributions around
$\Sigma$.

\begin{figure*}[t]
\centering
\begin{tikzpicture}[scale=1.05]

\node at (-2.0,3.4) {\large\textbf{Two-site case ($m=2$)}};

\fill[blue!10] (-4,0) rectangle (-2.1,0.5);
\fill[red!10]  (-2.1,0) rectangle (0,0.5);

\draw[thick] (-4,0) -- (0,0);

\filldraw[black] (-4,0) circle (2pt);
\filldraw[black] (0,0) circle (2pt);
\node[below=4pt] at (-4,0) {$\delta_{z_1}$};
\node[below=4pt] at (0,0)  {$\delta_{z_2}$};

\filldraw[black] (-2.1,0) circle (2.5pt);
\node[above=4pt] at (-2.1,0) {$\pi^*$};

\node[blue!70!black] at (-3.2,0.85) {$\partial_\gamma C < 0$};
\node[red!70!black]  at (-0.8,0.85) {$\partial_\gamma C > 0$};

\draw[->, thick, >=stealth] (-2.1,-0.75) -- (-0.9,-0.75);
\node[below=3pt] at (-1.5,-0.75) {$\psi(\gamma)$};

\node[align=center] at (-2.0,-2.0)
{\small $\operatorname{sgn}(\partial_\gamma C)
= \operatorname{sgn}\langle\pi-\pi^*,\psi(\gamma)\rangle$ (random walk)};

\node at (5.5,5.5) {\large\textbf{General case ($m\geq 3$)}};

\coordinate (A) at (3,0);
\coordinate (B) at (8,0);
\coordinate (C) at (5.5,4.33);

\fill[blue!12] (A) -- (4.2,2.1) -- (4.9,3.4) -- (5.2,4.0) -- (C) -- cycle;
\fill[gray!20] (4.2,2.1) -- (5.0,1.4) -- (6.0,3.0) -- (4.9,3.4) -- cycle;
\fill[red!12]  (5.0,1.4) -- (B) -- (C) -- (6.0,3.0) -- cycle;

\draw[thick] (A) -- (B) -- (C) -- cycle;

\node[below left=3pt]  at (A) {$\delta_{z_1}$};
\node[below right=3pt] at (B) {$\delta_{z_m}$};
\node[above=4pt]       at (C) {$\delta_{z_k}$};

\draw[thick, dashed, gray!60] (4.1,2.3) -- (5.8,1.2);
\draw[thick, black!80]        (4.5,2.8) -- (6.2,1.7);
\draw[thick, dashed, gray!60] (4.9,3.3) -- (6.6,2.2);

\node[right=2pt, font=\small] at (5.8,1.2) {$\gamma_1$};
\node[right=2pt, font=\small] at (6.2,1.7) {$\gamma_2$};
\node[right=2pt, font=\small] at (6.6,2.2) {$\gamma_3$};
\draw[->, thick, >=stealth] (6.85,1.2) -- (6.85,2.2);
\node[right=2pt, font=\small] at (6.85,1.7) {$\gamma\uparrow$};

\node[rotate=-30, font=\small] at (5.1,2.3)
{$\ker\psi(\gamma)$};

\node[blue!70!black, font=\small]  at (3.85,0.85) {decreasing};
\node[font=\small]                 at (5.25,2.05) {transition};
\node[red!70!black,  font=\small]  at (7.0, 0.85) {increasing};

\node[align=center, font=\small] at (5.5,-1.1)
{$\dim V = 1$: no rotation, two regions};
\node[align=center, font=\small] at (5.5,-1.7)
{$\dim V \geq 2$: rotating $\ker\psi(\gamma)$,
tripartite structure};

\node[align=center] at (5.5,-2.5)
{\small $\operatorname{sgn}(\partial_\gamma C)
= \operatorname{sgn}\langle\pi-\pi^*,\psi(\gamma)\rangle$ (random walk)};

\end{tikzpicture}
\caption{Geometric organization of the reset response landscape.
\textit{Left}: in the two-site case ($m=2$), the critical
distribution $\pi^*$ divides the one-dimensional simplex into
two regions of opposite sign of $\partial_\gamma C$, with
the direction $\psi(\gamma)$ pointing toward the increasing
region.
\textit{Right}: in higher-dimensional simplices ($m\geq 3$),
each $\gamma$ determines a hyperplane $\ker\psi(\gamma)$ that may
rotate with $\gamma$ when $\dim V \geq 2$,
generating decreasing, transition, and increasing regions; the
local separatrix $\Sigma_{\mathrm{loc}}=\bigcap_\gamma\ker\psi(\gamma)$
is their common intersection and does not depend on $\gamma$.
When $\dim V=1$ the hyperplane does not rotate, it coincides with
$\Sigma_{\mathrm{loc}}$, and only two
regions appear (as in Test~3).
The sign formula shown is observed numerically in the canonical
realization (Tests~2 and~3); it does not follow from
(S1)--(S4) alone (Section~\ref{sec:conjecture}).}
\label{fig:global_geometry}
\end{figure*}
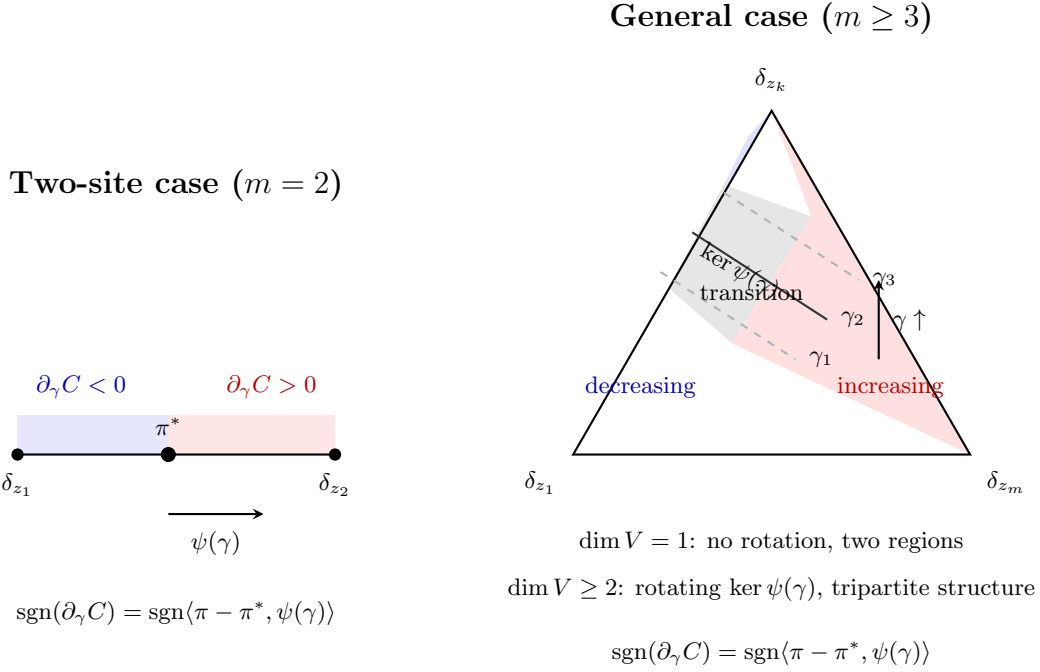

\newpage

\section{Global separatrix in the two-site case}\label{sec:twosite}

The geometric and algebraic structure established in Sections~\ref{sec:existence}--\ref{sec:geometric}
is in general formulated locally in the simplex. In the two-site
case ($m=2$), a stronger result holds: the separatrix is
\emph{global}, dividing the entire simplex into exactly two regions
of strictly opposite monotonic behavior. In higher-dimensional
simplices, the sign structure of $\partial_\gamma C$ may in
principle be considerably more intricate.

\medskip\noindent\textbf{Setting.} Consider a single spectral pair
$\{z, z'\}$ with $z' = \sigma(z) \neq z$. The simplex reduces to
the unit interval, parametrized by $\alpha = \pi_z \in [0,1]$
(with $\pi_{z'} = 1-\alpha$); the interior $\Delta_1^\circ =
(0,1)$ is where the main results apply. The coupling functional
takes the form
\[
C(\alpha,\gamma)
= \frac{\alpha\, u(z;\gamma) + (1-\alpha)\, u(z';\gamma)}
       {\alpha\, s(z;\gamma) + (1-\alpha)\, s(z';\gamma)},
\]
with denominator strictly positive by (S2).

\medskip\noindent\textbf{Assumption (M)} (Structural monotonicity).
The function
\[
R(\gamma) := \frac{u(z;\gamma)}{u(z';\gamma)}
\]
is strictly monotone in $\gamma$ on $(0,1)$, with $R'(\gamma)$
of constant non-zero sign.

\medskip\noindent\textbf{Remark 6.1.} Assumption~(M) is not part
of the abstract framework (S1)--(S4). It is an additional
structural property of the underlying process. For the biased
random walk, it follows from a Karlin--McGregor TP$_2$ argument
(Appendix~\ref{app:tp2}). Whether (M) holds for other realizations in the
class satisfying (S1)--(S4) is an open question.

\medskip\noindent\textbf{Theorem 6.2} (Global two-site separatrix).
\textit{Under assumptions (S1)--(S4) and Assumption~(M), the
critical point $\alpha^* \in (0,1)$ of Theorem~3.3 is the unique
global separatrix of the coupling functional on $\Delta_1^\circ$.
Specifically:}
\begin{itemize}
\item[(i)] \textit{$C(\alpha^*,\gamma) = C^*$ for all
      $\gamma \in (0,1)$.}
\item[(ii)] \textit{For $\alpha \neq \alpha^*$, the function
      $\gamma \mapsto C(\alpha,\gamma)$ is strictly monotone of
      constant sign, with the sign of $\partial_\gamma C$
      equal to $\mathrm{sgn}(\alpha-\alpha^*)\cdot\mathrm{sgn}(R'(\gamma))$.}
\end{itemize}
\textit{Moreover,}
\[
\begin{aligned}
\alpha^* &= \tfrac{\sqrt{\kappa(z')}}{\sqrt{\kappa(z)}+\sqrt{\kappa(z')}},\\
C^* &= \tfrac{1}{1+\sqrt{K}},\quad K = \kappa(z)\kappa(z').
\end{aligned}
\]

\textit{Proof.}
\textit{Step 1: existence and uniqueness of $\alpha^*$.}
By (S2) and (S3), for $i \in \{z, z'\}$:
\[
s_i = u_i + \kappa_i\, u_{\sigma(i)},
\]
since $s(z;\gamma) = \sum_\nu f_\nu(\gamma)(A_\nu(z)+B_\nu(z))$
and $B_\nu(z) = \kappa(z)A_\nu(\sigma(z))$ by (S3), so
$\sum_\nu f_\nu B_\nu(z) = \kappa(z)\sum_\nu f_\nu A_\nu(\sigma(z))
= \kappa(z)\,u(\sigma(z);\gamma)$.

By Lemma~3.2, the condition $C(\pi^*,\gamma)\equiv C^*$ reduces to
$U_\nu(\pi^*) = C^*S_\nu(\pi^*)$ for all $\nu$. For $m=2$, using
the relation above:
\[
\begin{aligned}
&(1-C^*)\bigl(\alpha^* u_1 + (1-\alpha^*)u_2\bigr)
= \\
&C^*(\alpha^*\kappa_1 u_2 + (1\!-\!\alpha^*)\kappa_2 u_1).
\end{aligned}
\]
Collecting terms in $u_1$ and $u_2$:
\[
\begin{aligned}
&\bigl[(1-C^*)\alpha^* - C^*(1-\alpha^*)\kappa_2\bigr]u_1\\
+&\bigl[(1-C^*)(1-\alpha^*) - C^*\alpha^*\kappa_1\bigr]u_2 = 0.
\end{aligned}
\]
Using the spectral representation (S2), this becomes
$\sum_\nu f_\nu(\gamma)[\ldots A_\nu(z) + \ldots A_\nu(z')] = 0$
for all $\gamma$. By linear independence of $\{f_\nu\}$, each
bracket vanishes for all $\nu$. By linear independence of
$(A_\nu(z))_\nu$ and $(A_\nu(z'))_\nu$ in $\mathbb{R}^N$ (both
from (S2)), the scalar coefficients must each be zero:
\[
\begin{cases}
(1-C^*)\alpha^* = C^*(1-\alpha^*)\kappa_2, \\
(1-C^*)(1-\alpha^*) = C^*\alpha^*\kappa_1.
\end{cases}
\]
Dividing: $(\alpha^*/(1-\alpha^*))^2 = \kappa_2/\kappa_1$, giving
$\alpha^* = \sqrt{\kappa_2}/(\sqrt{\kappa_1}+\sqrt{\kappa_2})$.
From the first equation:
$(1-C^*)/C^* = (1-\alpha^*)\kappa_2/\alpha^* =
(\sqrt{\kappa_1}/\sqrt{\kappa_2})\kappa_2 = \sqrt{\kappa_1\kappa_2}
= \sqrt{K}$, hence $C^* = 1/(1+\sqrt{K})$.

\textit{Step 2: strict monotonicity.}
A direct calculation (Appendix~\ref{app:factorisation}) gives the factorisation
\begin{equation}
\partial_\gamma C(\alpha,\gamma)
= (\alpha-\alpha^*)\,
\frac{G(\alpha)\,(u_1'u_2-u_1u_2')}
{[\alpha s_1+(1-\alpha)s_2]^2},
\end{equation}
where
$G(\alpha) := (\sqrt{\kappa_1}+\sqrt{\kappa_2})
\bigl(\sqrt{\kappa_1}\,\alpha+\sqrt{\kappa_2}(1-\alpha)\bigr)$
is strictly positive for $\alpha\in(0,1)$, and the denominator is
strictly positive by (S2). Under Assumption~(M),
$R(\gamma) = u_1/u_2$ is strictly monotone, so
$u_1'u_2 - u_1 u_2' = u_2^2 R'(\gamma)$ has constant non-zero
sign. Therefore the entire sign structure of $\partial_\gamma C$ is
encoded in the two simple factors $(\alpha-\alpha^*)$ and
$R'(\gamma)$:
\[
\mathrm{sgn}(\partial_\gamma C) =
\mathrm{sgn}(\alpha-\alpha^*)\cdot\mathrm{sgn}(R'),
\]
establishing part (ii). \hfill$\square$

\medskip\noindent\textbf{Remark 6.3} (Topological structure).
Theorem~6.2 reveals that the separatrix is a global topological
feature of $\Delta_1^\circ$. The interval $(0,1)$ is partitioned
by $\alpha^*$ into two open regions $(0,\alpha^*)$ and
$(\alpha^*,1)$, each with a constant sign of $\partial_\gamma C$.
The point $\alpha^*$ is the unique reset-neutral distribution.
The boundary points $\alpha \in \{0,1\}$ (degenerate
distributions) lie outside the interior analysis but can be
treated by continuity.

\medskip\noindent\textbf{Remark 6.4} (Canonical realization). For
the biased random walk (arXiv:2605.00657) with $\sigma(z) = a-z$
and $\kappa(z) = (p/q)^{a-z}$, one has $\kappa_1 = (p/q)^{a-z_1}$
and $\kappa_2 = (p/q)^{z_1}$, giving
$\sqrt{\kappa_2/\kappa_1} = (q/p)^{a/2-z_1}$ and
\[
\alpha^* = \frac{(q/p)^{a/2-z_1}}{1+(q/p)^{a/2-z_1}},
\]
recovering the two-site critical distribution of Paper~II. The
monotonicity of $R(\gamma)$ is verified via TP$_2$ in Appendix~\ref{app:tp2}.

\section{Canonical realization: the biased random walk}\label{sec:canonical}

We now verify that the biased random walk with multi-site geometric
resetting (Paper~II, arXiv:2605.00657) satisfies the structural
assumptions (S1)--(S4). Consequently, all results of Sections~\ref{sec:existence}--\ref{sec:twosite}
apply.

\subsection{Model and spectral decomposition}

Let $\{X_n\}_{n\ge 0}$ be a biased random walk on $\{0,1,\ldots,a\}$
with absorbing boundaries at $0$ and $a$, transition probabilities
$\mathbb{P}(X_{n+1} = x+1) = p$, $\mathbb{P}(X_{n+1} = x-1) = q
= 1-p$, and reset probability $\gamma \in (0,1)$. For
$z \in \{1,\ldots,a-1\}$, define
\[
\begin{aligned}
u(z;\gamma) &= \mathbb{E}_z\bigl[(1-\gamma)^{\tau_0}\mathbf{1}_{\{\tau_0<\tau_a\}}\bigr],\\
s(z;\gamma) &= \mathbb{E}_z[(1-\gamma)^\tau],\quad \tau = \min\{\tau_0,\tau_a\}.
\end{aligned}
\]
Here $u(z;\gamma)$ is the $\gamma$-discounted ruin probability (probability of absorption at $0$ before $a$, discounted by $(1-\gamma)$ per step), and $s(z;\gamma)$ is the $\gamma$-discounted absorption probability (at either boundary).
Via a Doob $h$-transform
with $h(x) = (q/p)^{x/2}$ (see Paper~II for details), one obtains
the spectral decomposition, valid for all $z\in\{1,\dots,a-1\}$,
\begin{align*}
u(z;\gamma) &= \textstyle\sum_{\nu=1}^{a-1} f_\nu(\gamma)A_\nu(z),\\
s(z;\gamma) &= \textstyle\sum_{\nu=1}^{a-1} f_\nu(\gamma)(A_\nu(z)+B_\nu(z)).
\end{align*}
where
\begin{align*}
f_\nu(\gamma) &= \frac{1-\gamma}{1-\lambda_\nu(1-\gamma)},\quad
\lambda_\nu = 2\sqrt{pq}\cos\tfrac{\pi\nu}{a}.
\end{align*}
\begin{align*}
A_\nu(z) &= \tfrac{2\sqrt{pq}}{a}\bigl(\tfrac{q}{p}\bigr)^{z/2}
\sin\tfrac{\pi\nu z}{a}\sin\tfrac{\pi\nu}{a},\\
B_\nu(z) &= \tfrac{2\sqrt{pq}}{a}\bigl(\tfrac{p}{q}\bigr)^{(a-z)/2}
\sin\tfrac{\pi\nu(a-z)}{a}\sin\tfrac{\pi\nu}{a}.
\end{align*}

\subsection{Verification of (S1)--(S4)}

\begin{itemize}
\item[\textbf{(S1)}] \textit{Regularity.} $C(\pi,\gamma)$ is
      $\mathcal{C}^1$ in $\pi$ since $u$ and $s$ are smooth in
      $\gamma$ and depend linearly on $\pi$.

\item[\textbf{(S2)}] \textit{Spectral representability.} The
      decomposition above is finite with $N = a-1$ modes. The
      functions $\{f_\nu\}_{\nu=1}^{a-1}$ are linearly independent
      on $(0,1)$: the values $\lambda_\nu = 2\sqrt{pq}
      \cos(\pi\nu/a)$ are strictly decreasing in $\nu$ and hence
      distinct, so a vanishing linear combination of the $f_\nu$
      implies all coefficients are zero. For distinct reset sites
      $z_1,\ldots,z_m$, the vectors
      $(A_\nu(z_i))_{\nu=1}^{a-1} \in \mathbb{R}^{a-1}$
      are linearly independent: they are columns of a submatrix
      of the discrete sine transform, which has full column rank.
      Here $m \le a-1$, since the reset sites form a subset of the
      interior $\{1,\ldots,a-1\}$.

\item[\textbf{(S3)}] \textit{Spectral duality.} Set $\sigma(z) =
      a-z$ and $\kappa(z) = (p/q)^{a-z}$. The trigonometric
      identity $\sin(\pi\nu(a-z)/a) = (-1)^{\nu+1}\sin(\pi\nu z/a)$
      gives
      \[
      B_\nu(z) = \kappa(z)\,A_\nu(\sigma(z)).
      \]
      The weights $\kappa(z) = (p/q)^{a-z}$ are independent of
      $\nu$.

\item[\textbf{(S4)}] \textit{Compatibility.}
      $\kappa(z)\kappa(\sigma(z)) = (p/q)^{a-z}(p/q)^{z}
      = (p/q)^a =: K$.
      When $a$ is even, the midpoint $z_0 = a/2$ is a fixed
      point of $\sigma$ (i.e., $\sigma(z_0) = z_0$), and
      $\kappa(z_0) = (p/q)^{a/2} = \sqrt{K}$, as required
      by (S4).
\end{itemize}

Thus the biased random walk satisfies (S1)--(S4).

\subsection{Consequences of the abstract theory}

Applying Theorem~3.3, the critical reset distribution $\pi^*$
satisfies
\[
\frac{\pi^*_z}{\pi^*_{a-z}}
= \Bigl(\frac{q}{p}\Bigr)^{\!a/2-z},
\]
and the invariant value is
\[
C^* = \frac{1}{1+\sqrt{K}}
= \frac{(q/p)^{a/2}}{1+(q/p)^{a/2}}
= q^{(0)}_{a/2}(p),
\]
the classical ruin probability from the midpoint, as established
in Paper~II. Proposition~5.1 identifies $\Sigma$ as the affine
simplex of distributions satisfying the ratio constraints above,
and Theorem~6.2 guarantees that in the two-site case $\Sigma$
acts as a global separatrix dividing $(0,1)$ into two regions
of opposite monotonicity.

\medskip\noindent\textbf{Remark 7.1.} The verification of
Assumption~(M) (strict monotonicity of $R(\gamma) =
u(z;\gamma)/u(z';\gamma)$) for the random walk is non-trivial.
It follows from the Karlin--McGregor total positivity property
of birth--death first-passage distributions and is carried out
in Appendix~\ref{app:tp2}.

\medskip\noindent\textbf{Remark 7.2} (Beyond geometric resetting).
The abstract framework enters through the weights $f_\nu(\gamma)$ of
geometric resetting. In the canonical realization, however, the
duality (S3) also holds step by step in time: the probabilities of
absorption at $a$ from $z$ and at $0$ from $a-z$ at each step $k$
differ exactly by the factor $\kappa(z)=(p/q)^{a-z}$~\cite{PaperI}.
As a consequence, a walk started from a distribution of $\Sigma$ and
reset to it has ruin probability $C^*$ under every resetting mechanism
independent of the walk under which absorption occurs almost surely,
not only under geometric resetting~\cite{PaperI,PaperII}. In the
canonical realization the separatrix $\Sigma$ of Proposition~5.1 is
therefore the set of reset distributions that are neutral for every
such mechanism; the single-site midpoint invariance of~\cite{PaperI}
is the case $\pi^*=\delta_{a/2}$.


\section{Numerical illustrations}\label{sec:numerics}

We present numerical experiments that validate the theoretical
predictions of the previous sections. All computations use the
biased random walk (Section~\ref{sec:canonical}) with parameters $a=10$, $p=0.6$.
Tests~1 and~2 use reset sites $\{3,7\}$ (two-site case);
Test~3 uses reset sites $\{3,5,7\}$ (three-site case, with
$z_2=5$ a fixed point of $\sigma$).

\subsection{Reset-neutral invariance: the separatrix}

Figure~\ref{fig:separatrix_2site} illustrates the defining property
of the critical distribution $\pi^*$: the coupling functional
$C(\alpha^*,\gamma)$ is exactly constant across all $\gamma\in(0,1)$,
confirming Theorem~3.3 to machine precision ($|C(\alpha^*,\gamma)
- C^*| < 10^{-10}$ for all tested values). Distributions with
$\alpha<\alpha^*$ produce a decreasing $C(\alpha,\gamma)$, while
those with $\alpha>\alpha^*$ produce an increasing one, consistent
with the global sign principle of Section~\ref{sec:twosite}.

\begin{figure}[H]
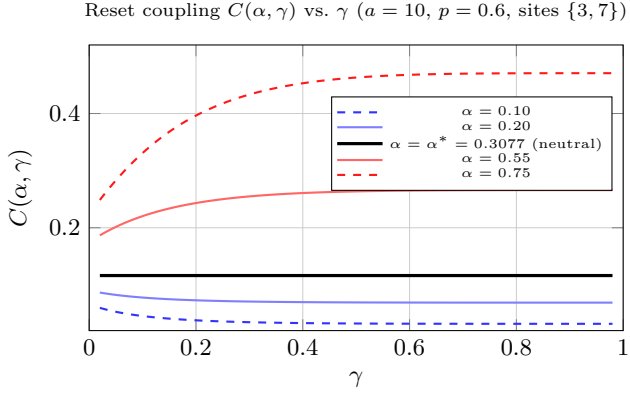

\centering


\caption{The coupling functional $C(\alpha,\gamma)$ as a function
of $\gamma$ for five values of $\alpha$, with reset sites
$\{z_1,z_2\}=\{3,7\}$, $a=10$, $p=0.6$.
The critical distribution $\alpha^*=0.3077$ produces a
\emph{horizontal} curve $C\equiv C^*=0.1164$ (thick black line),
confirming that $\pi^*$ is reset-neutral for all $\gamma\in(0,1)$.
Distributions with $\alpha<\alpha^*$ (blue) are decreasing in
$\gamma$; those with $\alpha>\alpha^*$ (red) are increasing. The
critical distribution thus acts as a global separatrix between
the two monotone regimes, validating not merely reset-neutrality
but the existence of a global boundary in the projective space
of reset distributions.}
\label{fig:separatrix_2site}
\end{figure}

\subsection{Test 1: $\psi(\gamma)$ as a
privileged direction}

Let $\pi^*$ be the critical distribution ($\alpha^*=0.30769$)
and $\psi(\gamma)$ the linear response functional. In the
two-site case the tangent space of $\Delta_1$ is one-dimensional,
so $\psi(\gamma)$ is the unique tangential direction; the test
measures the accuracy of the linear approximation. With
$\psi(\gamma)$ represented by the tangent vector of
\eqref{eq:psi_explicit} and $\hat\psi=\psi/\|\psi\|$, for a
perturbation of Euclidean size $\epsilon$ we compute the ratios
$\pm[C(\pi^*\pm\epsilon\hat\psi,\gamma)-C^*]/(\epsilon\|\psi\|)$,
which tend to $1$ as $\epsilon\to 0$ (here $\|\psi\|=0.3248$).
The result is shown in Figure~\ref{fig:test1_psi}.

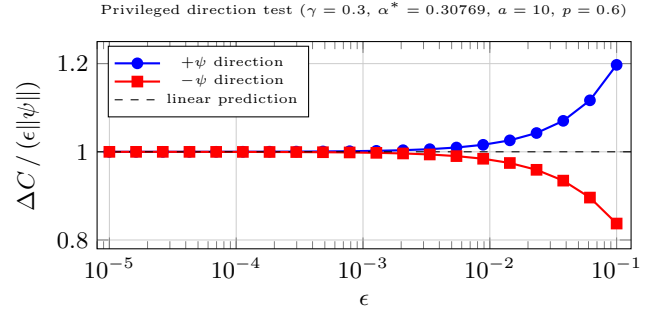
\begin{figure}[H]
\centering
\begin{tikzpicture}
\begin{axis}[
    width=\columnwidth,
    height=0.50\columnwidth,
    xlabel={$\epsilon$},
    ylabel={$\Delta C\,/\,(\epsilon\|\psi\|)$},
    xmode=log,
    xmin=8e-6, xmax=1.3e-1,
    ymin=0.92, ymax=1.08,
    grid=major,
    grid style={line width=0.3pt, draw=gray!40},
    legend style={font=\tiny, at={(0.02,0.98)}, anchor=north west,
      inner sep=2pt, row sep=-2pt},
    title style={font=\tiny},
    title={Privileged direction test
      ($\gamma=0.3$,\ $\alpha^*=0.30769$,\ $a=10$,\ $p=0.6$)},
]
\addplot[blue, thick, mark=*, mark size=1.8pt] coordinates {
(1.000000e-05,1.000006)(1.623777e-05,1.000009)(2.636651e-05,1.000015)
(4.281332e-05,1.000025)(6.951928e-05,1.000040)(1.128838e-04,1.000065)
(1.832981e-04,1.000105)(2.976351e-04,1.000171)(4.832930e-04,1.000277)
(7.847600e-04,1.000450)(1.274275e-03,1.000731)(2.069138e-03,1.001188)
(3.359818e-03,1.001931)(5.455595e-03,1.003139)(8.858668e-03,1.005108)
(1.438450e-02,1.008320)(2.335721e-02,1.013580)(3.792690e-02,1.022240)
(6.158482e-02,1.036621)(1.000000e-01,1.060854)
};
\addlegendentry{$+\hat\psi$ direction}
\addplot[red, thick, mark=square*, mark size=1.8pt] coordinates {
(1.000000e-05,0.999994)(1.623777e-05,0.999991)(2.636651e-05,0.999985)
(4.281332e-05,0.999975)(6.951928e-05,0.999960)(1.128838e-04,0.999935)
(1.832981e-04,0.999895)(2.976351e-04,0.999829)(4.832930e-04,0.999723)
(7.847600e-04,0.999550)(1.274275e-03,0.999270)(2.069138e-03,0.998814)
(3.359818e-03,0.998076)(5.455595e-03,0.996880)(8.858668e-03,0.994944)
(1.438450e-02,0.991816)(2.335721e-02,0.986779)(3.792690e-02,0.978707)
(6.158482e-02,0.965878)(1.000000e-01,0.945749)
};
\addlegendentry{$-\hat\psi$ direction}
\addplot[black, dashed, thin, domain=8e-6:1.3e-1] {1};
\addlegendentry{linear prediction}
\end{axis}
\end{tikzpicture}
\caption{Relative variation of $C$ in the directions
$\pm\hat\psi(\gamma)$. Both ratios converge to $1$ as $\epsilon\to 0$,
confirming that $\psi(\gamma)$ is a privileged direction for which
the linear approximation is highly accurate. The asymmetry for
large $\epsilon$ reflects the nonlinear dependence of $C$ on $\pi$.}
\label{fig:test1_psi}
\end{figure}

\subsection{Test 2: sign of the response and the linear functional}

We compare the sign of $\partial_\gamma C$, which Theorem~6.2 gives
as $\operatorname{sgn}(\alpha-\alpha^*)\operatorname{sgn}R'$, with the
sign of $\langle\pi-\pi^*,\psi(\gamma)\rangle$ over the entire
interval $\alpha\in(0,1)$, at $\gamma=0.3$; the result is
shown in Figure~\ref{fig:test2_theoremD}.

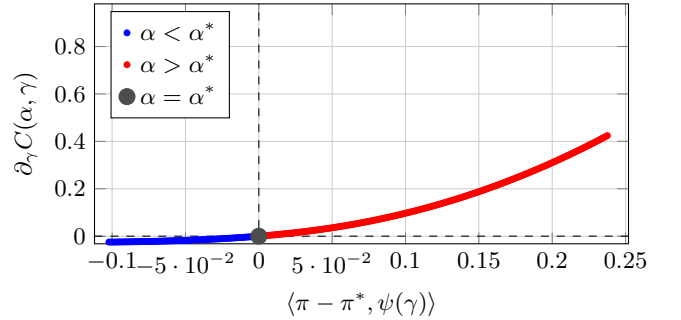
\begin{figure}[H]
\centering
\begin{tikzpicture}
\begin{axis}[
    width=\columnwidth,
    height=0.55\columnwidth,
    xlabel={$\langle\pi-\pi^*,\psi(\gamma)\rangle$},
    ylabel={$\partial_\gamma C(\alpha,\gamma)$},
    xmin=-0.155, xmax=0.335,
    ymin=-0.06, ymax=1.06,
    grid=major,
    grid style={line width=0.3pt, draw=gray!40},
    legend pos=north west,
    legend style={font=\small},
    title style={font=\scriptsize},
    title={Sign of the response
      ($\gamma=0.3$,\ $a=10$,\ $p=0.6$,\ sites $\{3,7\}$)},
]
\addplot[blue, only marks, mark=*, mark size=1.1pt] coordinates {
(-0.13216,-0.02485)(-0.13069,-0.02477)(-0.12921,-0.02470)
(-0.12774,-0.02461)(-0.12626,-0.02453)(-0.12479,-0.02444)
(-0.12331,-0.02435)(-0.12184,-0.02425)(-0.12036,-0.02415)
(-0.11889,-0.02405)(-0.11741,-0.02394)(-0.11594,-0.02383)
(-0.11446,-0.02372)(-0.11299,-0.02361)(-0.11151,-0.02349)
(-0.11004,-0.02336)(-0.10856,-0.02324)(-0.10709,-0.02311)
(-0.10561,-0.02297)(-0.10414,-0.02283)(-0.10266,-0.02269)
(-0.10119,-0.02255)(-0.09971,-0.02240)(-0.09824,-0.02224)
(-0.09676,-0.02209)(-0.09529,-0.02193)(-0.09381,-0.02176)
(-0.09234,-0.02159)(-0.09086,-0.02142)(-0.08939,-0.02124)
(-0.08791,-0.02106)(-0.08644,-0.02087)(-0.08496,-0.02068)
(-0.08349,-0.02048)(-0.08201,-0.02028)(-0.08054,-0.02008)
(-0.07906,-0.01987)(-0.07759,-0.01965)(-0.07612,-0.01944)
(-0.07464,-0.01921)(-0.07317,-0.01898)(-0.07169,-0.01875)
(-0.07022,-0.01851)(-0.06874,-0.01827)(-0.06727,-0.01802)
(-0.06579,-0.01777)(-0.06432,-0.01751)(-0.06284,-0.01724)
(-0.06137,-0.01697)(-0.05989,-0.01670)(-0.05842,-0.01642)
(-0.05694,-0.01613)(-0.05547,-0.01584)(-0.05399,-0.01554)
(-0.05252,-0.01524)(-0.05104,-0.01493)(-0.04957,-0.01462)
(-0.04809,-0.01430)(-0.04662,-0.01397)(-0.04514,-0.01363)
(-0.04367,-0.01330)(-0.04219,-0.01295)(-0.04072,-0.01260)
(-0.03924,-0.01224)(-0.03777,-0.01187)(-0.03629,-0.01150)
(-0.03482,-0.01112)(-0.03334,-0.01074)(-0.03187,-0.01035)
(-0.03039,-0.00995)(-0.02892,-0.00954)(-0.02744,-0.00913)
(-0.02597,-0.00870)(-0.02449,-0.00828)(-0.02302,-0.00784)
(-0.02154,-0.00740)(-0.02007,-0.00695)(-0.01859,-0.00649)
(-0.01712,-0.00602)(-0.01564,-0.00555)(-0.01417,-0.00506)
(-0.01269,-0.00457)(-0.01122,-0.00408)(-0.00974,-0.00357)
(-0.00827,-0.00305)(-0.00679,-0.00253)(-0.00532,-0.00200)
(-0.00384,-0.00145)(-0.00237,-0.00090)(-0.00089,-0.00034)
};
\addlegendentry{$\alpha<\alpha^*$}
\addplot[red, only marks, mark=*, mark size=1.1pt] coordinates {
(0.00058,0.00023)(0.00206,0.00080)(0.00353,0.00139)
(0.00501,0.00199)(0.00648,0.00260)(0.00796,0.00321)
(0.00943,0.00384)(0.01091,0.00448)(0.01238,0.00512)
(0.01386,0.00578)(0.01533,0.00645)(0.01681,0.00713)
(0.01828,0.00782)(0.01976,0.00852)(0.02123,0.00923)
(0.02271,0.00995)(0.02418,0.01068)(0.02566,0.01143)
(0.02713,0.01218)(0.02861,0.01295)(0.03008,0.01373)
(0.03156,0.01452)(0.03303,0.01533)(0.03451,0.01615)
(0.03598,0.01698)(0.03746,0.01782)(0.03893,0.01867)
(0.04041,0.01954)(0.04188,0.02043)(0.04336,0.02132)
(0.04483,0.02223)(0.04631,0.02315)(0.04778,0.02409)
(0.04926,0.02504)(0.05073,0.02601)(0.05221,0.02699)
(0.05368,0.02799)(0.05516,0.02900)(0.05663,0.03002)
(0.05811,0.03107)(0.05958,0.03212)(0.06106,0.03320)
(0.06253,0.03429)(0.06401,0.03539)(0.06548,0.03652)
(0.06696,0.03766)(0.06843,0.03881)(0.06991,0.03999)
(0.07138,0.04118)(0.07286,0.04239)(0.07433,0.04362)
(0.07581,0.04486)(0.07728,0.04613)(0.07876,0.04741)
(0.08023,0.04872)(0.08171,0.05004)(0.08318,0.05138)
(0.08466,0.05275)(0.08613,0.05413)(0.08761,0.05553)
(0.08908,0.05696)(0.09056,0.05841)(0.09203,0.05987)
(0.09351,0.06136)(0.09498,0.06288)(0.09646,0.06441)
(0.09793,0.06597)(0.09941,0.06755)(0.10088,0.06915)
(0.10236,0.07078)(0.10383,0.07244)(0.10531,0.07411)
(0.10678,0.07582)(0.10826,0.07755)(0.10973,0.07930)
(0.11121,0.08108)(0.11268,0.08289)(0.11416,0.08472)
(0.11563,0.08658)(0.11711,0.08847)(0.11858,0.09039)
(0.12006,0.09234)(0.12153,0.09431)(0.12301,0.09632)
(0.12448,0.09836)(0.12596,0.10042)(0.12743,0.10252)
(0.12891,0.10465)(0.13038,0.10681)(0.13186,0.10900)
(0.13333,0.11123)(0.13481,0.11349)(0.13628,0.11579)
(0.13776,0.11812)(0.13923,0.12048)(0.14071,0.12288)
(0.14218,0.12532)(0.14366,0.12779)(0.14513,0.13031)
(0.14661,0.13286)(0.14808,0.13545)(0.14956,0.13807)
(0.15103,0.14074)(0.15251,0.14345)(0.15398,0.14620)
(0.15546,0.14900)(0.15693,0.15183)(0.15841,0.15471)
(0.15988,0.15764)(0.16136,0.16060)(0.16283,0.16362)
(0.16431,0.16668)(0.16578,0.16979)(0.16726,0.17294)
(0.16873,0.17615)(0.17021,0.17941)(0.17168,0.18271)
(0.17316,0.18607)(0.17463,0.18948)(0.17611,0.19294)
(0.17758,0.19646)(0.17906,0.20003)(0.18053,0.20366)
(0.18201,0.20735)(0.18348,0.21109)(0.18496,0.21489)
(0.18643,0.21876)(0.18791,0.22268)(0.18938,0.22667)
(0.19086,0.23072)(0.19233,0.23483)(0.19381,0.23901)
(0.19528,0.24326)(0.19676,0.24758)(0.19823,0.25196)
(0.19971,0.25642)(0.20118,0.26094)(0.20266,0.26554)
(0.20413,0.27022)(0.20561,0.27497)(0.20708,0.27980)
(0.20856,0.28470)(0.21003,0.28969)(0.21151,0.29476)
(0.21298,0.29991)(0.21446,0.30515)(0.21593,0.31047)
(0.21741,0.31588)(0.21888,0.32138)(0.22036,0.32698)
(0.22183,0.33266)(0.22331,0.33844)(0.22478,0.34432)
(0.22626,0.35029)(0.22773,0.35637)(0.22921,0.36255)
(0.23068,0.36883)(0.23216,0.37522)(0.23363,0.38172)
(0.23511,0.38833)(0.23658,0.39505)(0.23806,0.40189)
(0.23953,0.40884)(0.24101,0.41592)(0.24248,0.42311)
(0.24396,0.43043)(0.24543,0.43788)(0.24691,0.44546)
(0.24838,0.45317)(0.24986,0.46102)(0.25133,0.46901)
(0.25281,0.47713)(0.25428,0.48540)(0.25576,0.49382)
(0.25723,0.50239)(0.25871,0.51111)(0.26018,0.51998)
(0.26166,0.52902)(0.26313,0.53822)(0.26460,0.54759)
(0.26608,0.55712)(0.26755,0.56684)(0.26903,0.57672)
(0.27050,0.58679)(0.27198,0.59705)(0.27345,0.60750)
(0.27493,0.61814)(0.27640,0.62897)(0.27788,0.64001)
(0.27935,0.65126)(0.28083,0.66272)(0.28230,0.67440)
(0.28378,0.68630)(0.28525,0.69842)(0.28673,0.71078)
(0.28820,0.72337)(0.28968,0.73621)(0.29115,0.74929)
(0.29263,0.76263)(0.29410,0.77623)(0.29558,0.79010)
(0.29705,0.80424)(0.29853,0.81865)(0.30000,0.83336)
(0.30148,0.84835)(0.30295,0.86365)(0.30443,0.87925)
(0.30590,0.89517)(0.30738,0.91141)(0.30885,0.92798)
};
\addlegendentry{$\alpha>\alpha^*$}
\addplot[black!70, only marks, mark=*, mark size=3pt]
    coordinates {(0,0)};
\addlegendentry{$\alpha=\alpha^*$}
\draw[black, dashed, thin] (axis cs:-0.155,0)--(axis cs:0.335,0);
\draw[black, dashed, thin] (axis cs:0,-0.06)--(axis cs:0,1.06);
\end{axis}
\end{tikzpicture}
\caption{Canonical realization: $\partial_\gamma C$ vs
$\langle\pi-\pi^*,\psi(\gamma)\rangle$ for $300$ values of
$\alpha\in(0,1)$ at $\gamma=0.3$. Sign agreement is exact
(100\%). The sign changes precisely at $\alpha=\alpha^*$, where
both quantities vanish. The nonlinear shape reflects higher-order
contributions far from $\alpha^*$.}
\label{fig:test2_theoremD}
\end{figure}

This agreement is a property of the canonical realization, not a
consequence of Theorem~6.2 or of (S1)--(S4); see
Section~\ref{sec:conjecture}.

\subsection{Phase structure for $m=3$ reset sites}

For $m=3$ with sites $\{3,5,7\}$, the site $z_2=5$ is a fixed
point of $\sigma$ and $\dim V=1$. As shown in
Figure~\ref{fig:test3_tripartite}, the simplex partitions into
exactly two connected regions separated by a straight line
(the separatrix $\Sigma$) running from vertex $z_2=5$ to the
point $(\alpha^*,1-\alpha^*,0)$ on the base, with
$\alpha^*=0.3077$. This behaviour is consistent with the
interpretation of $r=\dim V$ as a complexity index for the
response landscape: when $r=1$, the simplex exhibits a simple
bipartition into two monotone regions separated by a single
separatrix, connecting the integer $r$ directly to a geometric
feature the reader can see.

\begin{figure}[H]
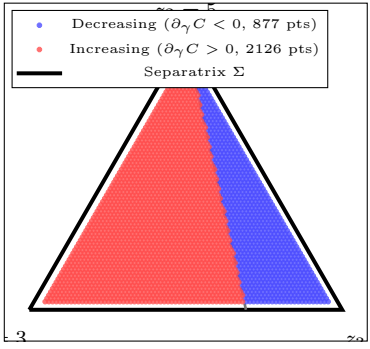

\centering

\caption{Phase partition of $\Delta_2$ for reset sites $\{3,5,7\}$.
Blue: decreasing region ($\partial_\gamma C<0$, 877 points).
Red: increasing region ($\partial_\gamma C>0$, 2126 points).
Both regions are connected and separated by the dashed line
(separatrix $\Sigma$), which runs from vertex $z_2=5$ to the
point $(\alpha^*,1-\alpha^*,0)$ with $\alpha^*=0.3077$.
No transition region appears (0 points), consistent with
$\dim V=1$.}
\label{fig:test3_tripartite}
\end{figure}

\subsection{Quantitative validation}

Table~\ref{tab:validation} compares the theoretical values of
$\alpha^*$ and $C^*$ with their numerically computed
counterparts across five representative parameter
configurations.

\begin{table}[H]
\centering
\resizebox{\columnwidth}{!}{%
\begin{tabular}{ccccccc}
\toprule
$a$ & $p$ & $z_1$ & $z_2$ & $\alpha^*_{\rm th}$ &
$C^*_{\rm th}$ & $|C^*_{\rm num}-C^*_{\rm th}|$ \\
\midrule
$10$ & $0.6$ & $3$ & $7$ & $0.30769231$ & $0.11636364$ &
$4.2\times10^{-17}$ \\
$10$ & $0.6$ & $2$ & $8$ & $0.22857143$ & $0.11636364$ &
$1.4\times10^{-17}$ \\
$10$ & $0.7$ & $3$ & $7$ & $0.15517241$ & $0.01425220$ &
$0$ \\
$10$ & $0.5$ & $3$ & $7$ & $0.50000000$ & $0.50000000$ &
$0$ \\
$8$  & $0.6$ & $2$ & $6$ & $0.30769231$ & $0.16494845$ &
$2.8\times10^{-17}$ \\
\bottomrule
\end{tabular}
}
\caption{Comparison of theoretical and numerical values of
$\alpha^*$ and $C^*$. The numerical $C^*$ is evaluated at
$\gamma=0.5$. All discrepancies are at the level of machine precision.}
\label{tab:validation}
\end{table}

\subsection{Summary of numerical findings}

The numerical experiments confirm the theoretical predictions:
\begin{itemize}
\item $\psi(\gamma)$ is a privileged direction with linear
      approximation error tending to zero as $\epsilon\to 0$,
      for perturbations in both directions (Test~1).
\item In the canonical realization, the sign of
      $\partial_\gamma C$ coincides with
      $\operatorname{sgn}\langle\pi-\pi^*,\psi(\gamma)\rangle$
      over all sampled points, with a unique zero at $\alpha^*$
      (Test~2). Beyond sign agreement, the data collapse onto a
      single monotone curve, suggesting that, in this realization,
      the projection $\langle\pi-\pi^*,\psi(\gamma)\rangle$
      provides an effective one-dimensional coordinate
      controlling the response.
\item For $m=3$ with sites $\{3,5,7\}$, the sampled simplex
      exhibits two connected regions of opposite sign separated
      by a straight-line separatrix, with no transition points
      detected; this behaviour is consistent with the prediction
      $\dim V=1$ (Test~3). It is in turn consistent with the
      interpretation of $r=\dim V$ as a complexity index for the
      response landscape: when $r=1$, the simplex exhibits a
      simple bipartition into two monotone regions separated by
      a single separatrix.
\item Theoretical values of $\alpha^*$ and $C^*$ are confirmed
      to machine precision across all tested parameters
      (Table~\ref{tab:validation}).
\end{itemize}
These results provide strong empirical support for the geometric
framework of Sections~\ref{sec:setting}--\ref{sec:twosite}.
The agreement between the sign of the response and the linear
functional observed in Tests~2 and~3 is, however, a property of the
canonical realization and not of (S1)--(S4); see
Section~\ref{sec:conjecture}.

\section{Global orientation of the response}
\label{sec:conjecture}

The numerical experiments of Section~\ref{sec:numerics} show that,
in the canonical realization, the sign of $\partial_\gamma C$ is
given globally by the linear functional at the separatrix:
\[
\operatorname{sgn}\bigl(\partial_\gamma C(\pi,\gamma)\bigr)
= \operatorname{sgn}\langle\pi - \pi^*, \psi(\gamma)\rangle .
\]
This global orientation does not follow from (S1)--(S4): explicit
counterexamples exist that satisfy (S1)--(S4), and even the two-site
structural monotonicity of Assumption~(M), and for which the two signs
above are opposite for every $\gamma\in(0,1)$ and every
$\alpha\neq\alpha^*$~\cite{NewtonOrientation}. Theorem~6.2 is not
affected: it gives the sign of $\partial_\gamma C$ as
$\operatorname{sgn}(\alpha-\alpha^*)\operatorname{sgn}R'$, and it is
correct in those examples as well. What fails is the identification of
this sign with that of $\langle\pi-\pi^*,\psi(\gamma)\rangle$, which
requires information on the underlying dynamics that (S1)--(S4) do not
contain. The orientation of the response is thus dynamical rather than
algebraic information. We thank T.~Newton for these counterexamples,
and refer to~\cite{NewtonOrientation} for their construction and for
the conditions under which the two signs agree.

\medskip\noindent\textbf{Numerical evidence in the canonical
realization.}
In the realizations of the biased random walk the two criteria do
agree. Test~2 of Section~\ref{sec:numerics} finds exact sign agreement
over 300 uniformly sampled distributions $\pi\in\Delta_1^\circ$
(100\%), and Test~3 finds zero transition points over 3003 sampled
distributions in $\Delta_2^\circ$ ($m=3$, $\dim V=1$), consistent with
a clean bipartition of the simplex by $\Sigma$. These results hold
across all tested parameter values ($a\in\{8,10\}$,
$p\in\{0.5, 0.6, 0.7\}$, multiple site configurations).
Characterizing the realizations in which the global orientation holds
remains an open problem.

\section{Conclusions and perspectives}\label{sec:conclusions}

This paper has developed an abstract framework for the study of
reset-neutral distributions in absorbed Markov processes with
geometric resetting. Starting from four structural conditions
(S1)--(S4) on the coupling functional $C(\pi,\gamma)$, we have
established the existence of a critical manifold
$\Sigma\subset\Delta_{m-1}$, characterized its geometry, and
identified the mechanism by which the spectral duality organizes
the reset response landscape.

\medskip\noindent\textbf{Summary of main results.}
Under assumptions (S1)--(S4), the following hold.

\begin{itemize}
\item[(i)] There exists a non-empty family of reset-neutral
distributions $\pi^*\in\Delta_{m-1}$ such that
$C(\pi^*,\gamma) = C^* = 1/(1+\sqrt{K})$ for all
$\gamma\in(0,1)$, where $K$ is the structural compatibility
constant. The value $C^*$ depends only on $K$, not on the
specific choice of $\pi^*$ nor on the neutral weights
(Theorem~3.3).

\item[(ii)] The separatrix $\Sigma$ is an affine simplex of
dimension $n_{\rm pair}+n_{\rm neutral}-1$, determined by explicit
ratio constraints on the reset weights (Proposition~5.1).
Perturbations tangential to $\Sigma$ preserve reset-neutrality
to first order, while transversal perturbations destroy it for every
small displacement
(Propositions~5.7--5.8).

\item[(iii)] Under an additional parity condition, the spectral
coefficients $[A_\nu(z):B_\nu(z)]$ are confined to two
points $[1:\epsilon_\nu\kappa(z)]$ in
$\mathbb{P}^1$ (in the canonical realization, to the two privileged
points $[1:\pm\sqrt K]$), and $\pi^*$ aligns the spectral channels to
the ratio $\sqrt{K}$ in a weighted average sense
(Lemma~4.1, Proposition~4.3).

\item[(iv)] In the two-site case, under Assumption~(M), the
separatrix is global: it divides $\Delta_1^\circ$ into two
regions of strictly opposite monotonic behavior, with
$\operatorname{sgn}(\partial_\gamma C) =
\operatorname{sgn}(\alpha-\alpha^*)\cdot\operatorname{sgn}(R')$
(Theorem~6.2).

\item[(v)] The biased random walk with multi-site geometric
resetting satisfies (S1)--(S4) and Assumption~(M), recovering
the results of Papers~I--II as a canonical realization
(Section~\ref{sec:canonical}).
\end{itemize}

\medskip\noindent\textbf{The response landscape.}
The results above reveal that the simplex $\Delta_{m-1}$
carries a non-trivial \emph{spectral response geometry}
induced by the duality structure: a foliation by level sets
of $C$, organized around the critical manifold $\Sigma$,
with the linear response functional $\psi(\gamma)$ at $\Sigma$. In
the canonical realization, $\psi(\gamma)$ also orients the response
globally, $\operatorname{sgn}(\partial_\gamma C)=
\operatorname{sgn}\langle\pi-\pi^*,\psi(\gamma)\rangle$, as observed
numerically (Section~\ref{sec:numerics}); this global orientation
does not follow from (S1)--(S4) and is dynamical rather than algebraic
information (Section~\ref{sec:conjecture}).

\medskip\noindent\textbf{Perspectives.}
The present work raises several natural questions.

\textit{Operator-theoretic origin (Paper~IV~\cite{PaperIV}).} The spectral
duality (S3) and the associated geometry have a natural
interpretation in terms of the resolvent of the underlying
Markov process. The field $\psi(\gamma)$ is expected to admit
a resolvent representation, and the privileged points
$[1:\pm\sqrt{K}]$ in $\mathbb{P}^1$ should emerge from the
Perron--Frobenius structure of the generator. Making this
precise would explain why the twisted symmetry
$B_\nu(z) = \kappa(z)A_\nu(\sigma(z))$ arises in a broad
class of absorbed Markov chains with resetting.

\textit{Information geometry and non-equilibrium thermodynamics
(Paper~V~\cite{PaperV}).} The response landscape has a natural thermodynamic
interpretation. The flow
\[
\dot\pi = -\Pi_{T_\pi\Delta}\!\left(
\int\psi(\gamma)\,\partial_\gamma C(\pi,\gamma)\,d\gamma
\right)
\]
is induced directly by the observable $C$ and the response
field $\psi(\gamma)$. The functional
\[
F_3(\pi) = \int\langle\pi-\pi^*,\psi(\gamma)\rangle^2\,d\gamma
\]
emerges as a natural Lyapunov candidate from the structure
of the observable, not by postulation. Two concrete questions
make the connection with information geometry precise.
First, the Fisher--Rao metric on $\Delta_{m-1}$ is a canonical
Riemannian structure on the simplex; a key open question is
whether $\Sigma$ is a geodesic submanifold under this metric,
which would give the separatrix a thermodynamic interpretation
as a surface of minimal informational cost.
Second, $F_3(\pi)$ is reminiscent of relative entropy in
non-equilibrium systems: if $F_3$ decreases monotonically
along the flow, it would serve as a genuine Lyapunov function
with the interpretation that the system dissipates reset
response energy as it relaxes toward the neutral manifold~$\Sigma$.

\textit{Universality and classification.} Since the global
orientation of the response does not follow from (S1)--(S4)
(Section~\ref{sec:conjecture}), a natural question is to characterize
the realizations in which it holds. A classification of response landscapes in terms
of $\dim V = \dim\operatorname{span}\{\psi(\gamma)\}$ ---
two regions for $\dim V=1$, tripartite structure for
$\dim V\geq 2$ --- would provide a topological picture of
the geometry of reset distributions beyond the cases studied
here.


\appendix
\section{Technical proofs for the two-site case}
\label{app:twosite}

\subsection{Factorisation of $\partial_\gamma C$}
\label{app:factorisation}

We derive the factorisation used in the proof of Theorem~6.2.
With the notation of Section~\ref{sec:twosite}, write $u_1 = u(z;\gamma)$,
$u_2 = u(z';\gamma)$, $s_1 = s(z;\gamma)$, $s_2 = s(z';\gamma)$,
and
\[
C(\alpha,\gamma) = \frac{\alpha u_1 + (1-\alpha)u_2}
                       {\alpha s_1 + (1-\alpha)s_2}.
\]

Differentiating with respect to $\gamma$ via the quotient rule,
\[
\partial_\gamma C = \frac{N(\alpha,\gamma)}
{[\alpha s_1 + (1-\alpha)s_2]^2},
\]
where
{\small\begin{align*}
N &= \bigl(\alpha u_1'+(1-\alpha)u_2'\bigr)
\bigl(\alpha s_1+(1-\alpha)s_2\bigr) \\
&\quad - \bigl(\alpha u_1+(1-\alpha)u_2\bigr)
\bigl(\alpha s_1'+(1-\alpha)s_2'\bigr).
\end{align*}}

Expanding both products:
{\small\begin{align*}
N &= \alpha^2(u_1's_1-u_1s_1')
+ (1-\alpha)^2(u_2's_2-u_2s_2') \\
&\quad + \alpha(1-\alpha)(u_1's_2+u_2's_1-u_1s_2'-u_2s_1').
\end{align*}}

By (S2) and (S3), $s_i = u_i + \kappa_i u_{\sigma(i)}$ and
$s_i' = u_i' + \kappa_i u_{\sigma(i)}'$, with
$\kappa_1=\kappa(z)$, $\kappa_2=\kappa(z')$.

\medskip\noindent\textit{Diagonal terms.}
\[
u_i's_i - u_is_i' = \kappa_i(u_i'u_{\sigma(i)} - u_iu_{\sigma(i)}').
\]

\medskip\noindent\textit{Cross terms.}
\[
\begin{split}
u_1's_2&+u_2's_1-u_1s_2'-u_2s_1'\\
&= (u_1'u_2-u_1u_2') + (u_2'u_1-u_2u_1') = 0,
\end{split}
\]
since the two terms cancel. Substituting:
\[
N = \bigl[\alpha^2\kappa_1-(1-\alpha)^2\kappa_2\bigr]
    \bigl(u_1'u_2-u_1u_2'\bigr).
\]

\medskip\noindent\textit{Factorisation of the algebraic factor.}
Since $\alpha^* = \sqrt{\kappa_2}/(\sqrt{\kappa_1}+\sqrt{\kappa_2})$,
the factor $\alpha^2\kappa_1-(1-\alpha)^2\kappa_2$ factors as
\[
\begin{split}
\alpha^2\kappa_1-(1-\alpha)^2\kappa_2 =
{}&(\sqrt{\kappa_1}+\sqrt{\kappa_2})(\alpha-\alpha^*)\\
&\times(\sqrt{\kappa_1}\,\alpha+\sqrt{\kappa_2}(1-\alpha)).
\end{split}
\]
The factor $\bigl(\sqrt{\kappa_1}\,\alpha+\sqrt{\kappa_2}(1-\alpha)\bigr)$
is strictly positive for all $\alpha\in(0,1)$, as is
$(\sqrt{\kappa_1}+\sqrt{\kappa_2})$. Substituting this factorisation
into $N$ and dividing by the denominator gives

\begin{equation}
\partial_\gamma C(\alpha,\gamma) = (\alpha-\alpha^*)\,
\frac{G(\alpha)\,(u_1'u_2-u_1u_2')}
{[\alpha s_1+(1-\alpha)s_2]^2},
\end{equation}
with $G(\alpha) := (\sqrt{\kappa_1}+\sqrt{\kappa_2})
\bigl(\sqrt{\kappa_1}\,\alpha+\sqrt{\kappa_2}(1-\alpha)\bigr) > 0$.
Using $u_1'u_2-u_1u_2' = u_2^2R'(\gamma)$ with $R=u_1/u_2$, and
noting that $G(\alpha)$ and the
denominator are strictly positive, the sign of
$\partial_\gamma C$ reduces to
\[
\operatorname{sgn}(\partial_\gamma C) =
\operatorname{sgn}(\alpha-\alpha^*)\cdot\operatorname{sgn}(R'(\gamma)).
\]
Under Assumption~(M), $R'(\gamma)$ has constant nonzero sign, so
the entire sign structure of $\partial_\gamma C$ is encoded in the
monotonicity of $R$ and the sign of $\alpha-\alpha^*$.
\hfill$\square$

\subsection{Verification of Assumption~(M) via
Karlin--McGregor TP$_2$}
\label{app:tp2}

We verify that $R(\gamma)=u(z;\gamma)/u(z';\gamma)$ is
strictly monotone for the biased random walk of Section~\ref{sec:canonical}.

\begin{definition}
A family $\{f_\nu(\gamma)\}$ is \emph{totally positive of
order~2} (TP$_2$) if for all $\nu_1<\nu_2$ and
$\gamma_1<\gamma_2$,
\[
f_{\nu_1}(\gamma_1)f_{\nu_2}(\gamma_2)
-f_{\nu_1}(\gamma_2)f_{\nu_2}(\gamma_1) > 0.
\]
\end{definition}

\begin{lemma}[TP$_2$ of the spectral weights]
\label{lem:tp2_fnu}
The family $\{f_\nu(\gamma)\}_{\nu=1}^{a-1}$, with
$f_\nu(\gamma) = (1-\gamma)/(1-\lambda_\nu(1-\gamma))$,
is TP$_2$ as a function of $(\nu,\gamma)$.
\end{lemma}

\begin{proof}
Since $f_\nu(\gamma)$ is a fractional linear transform of
$\lambda_\nu(1-\gamma)$ that preserves order, and since
$\lambda_\nu = 2\sqrt{pq}\cos(\pi\nu/a)$ is strictly decreasing
in $\nu$, the kernel $(\nu,\gamma)\mapsto f_\nu(\gamma)$ is
TP$_2$ (see Karlin, 1968, Ch.~5). This can also be verified
directly: a computation gives
\begin{align*}
f_\nu'(\gamma)f_\mu(\gamma) &- f_\mu'(\gamma)f_\nu(\gamma) \\
&= \frac{(\lambda_\mu-\lambda_\nu)(1-\gamma)^2}
{\bigl[1-\lambda_\nu(1-\gamma)\bigr]^2\bigl[1-\lambda_\mu(1-\gamma)\bigr]^2},
\end{align*}
which has constant sign $-\operatorname{sgn}(\lambda_\nu-\lambda_\mu)$
for all $\gamma\in(0,1)$, confirming the TP$_2$ property.
\hfill$\square$
\end{proof}

\begin{proposition}[Monotonicity of $R(\gamma)$]
\label{prop:monotonicity_R}
For any two distinct reset sites $z\neq z'$, the ratio
$R(\gamma) = u(z;\gamma)/u(z';\gamma)$ is strictly monotone
on $(0,1)$.
\end{proposition}

\begin{proof}
The function $u(z;\gamma)=(1-\gamma)\,[(I-(1-\gamma)P)^{-1}b^{(0)}]_z$
is, up to the positive factor $(1-\gamma)$, the discounted Green
function of the absorbed birth--death chain evaluated at the
absorbing flux $b^{(0)}$. By the Karlin--McGregor theory of
birth--death processes~\cite{KarlinMcGregor1959,Karlin1968}, the
Green kernel of such a chain is totally positive of order two
(TP$_2$) jointly in the state variable and the resolvent
parameter; concretely, the kernel $(z,\gamma)\mapsto u(z;\gamma)$
satisfies, for all $z\neq z'$ and $\gamma\neq\gamma'$,
\[
\operatorname{sgn}\bigl[u(z;\gamma)u(z';\gamma')
-u(z;\gamma')u(z';\gamma)\bigr]
\]
constant (independent of the particular $\gamma,\gamma'$). A
kernel is TP$_2$ in $(z,\gamma)$ precisely when the ratio
$R(\gamma)=u(z;\gamma)/u(z';\gamma)$ is monotone in $\gamma$ for
every fixed pair $z\neq z'$; the sign of the $2\times2$ minor
fixes the direction of monotonicity. Hence $R(\gamma)$ is
strictly monotone on $(0,1)$.

We emphasize that monotonicity does \emph{not} follow from a
term-by-term sign argument on the spectral expansion
$u(z;\gamma)=\sum_\nu f_\nu(\gamma)A_\nu(z)$: the spatial Wronskians
$A_\nu(z)A_\mu(z')-A_\nu(z')A_\mu(z)$ do not all share a common
sign, so the individual terms of $R'(\gamma)$ need not. It is the
total positivity of the assembled Green kernel --- a property of
the full sum, not of its summands --- that yields the result.
\hfill$\square$
\end{proof}

\begin{remark}
For the two-site case with $z'=a-z$, a direct computation
using the explicit formulas of Section~\ref{sec:canonical} shows that $R(\gamma)$
is strictly increasing when $z<a/2$. Consequently, in Theorem~6.2,
$\operatorname{sgn}(\partial_\gamma C)=\operatorname{sgn}(\alpha-\alpha^*)$
when $\alpha$ is the weight of the site closer to $0$, in agreement
with Figure~\ref{fig:separatrix_2site}. We thank T.~Newton for
pointing out the direction of this
monotonicity~\cite{NewtonOrientation}.
\end{remark}

\begin{acknowledgments}

The author thanks the Universidad de Salamanca for institutional
support, and acknowledges stimulating discussions with colleagues
at various stages of this work. He is grateful to T.~Newton for a
careful independent examination of this paper, which led to the
corrections acknowledged in the text.

\end{acknowledgments}

\end{document}